\documentclass[11pt]{amsart}

\usepackage{amsmath}             
\usepackage{amssymb}             
\usepackage{amsfonts}            
\usepackage{mathrsfs}            
\usepackage{amsthm}
\usepackage{mathtools}
\usepackage{commath}
\usepackage{bm}
\usepackage{graphicx}
\usepackage{geometry}
\usepackage{float}
\usepackage{listings}
\usepackage{subfigure}
\usepackage{multirow}
\usepackage{diagbox}
\usepackage[export]{adjustbox}
\usepackage[all]{xy}
\usepackage{tikz-cd}
\usepackage{fancyhdr}
\usepackage{dsfont}
\usepackage{hyperref}
\usepackage{xcolor}
\usepackage{enumitem}
\DeclareMathAlphabet{\mathbbb}{U}{bbold}{m}{n}

\theoremstyle{definition}
\newtheorem{defi}{Definition}[section]

\theoremstyle{plain}
\newtheorem{theo}[defi]{Theorem}
\newtheorem{lemma}[defi]{Lemma}
\newtheorem{prop}[defi]{Proposition}
\newtheorem{cor}[defi]{Corollary}

\theoremstyle{remark}
\newtheorem*{rmk}{Remark}

\title{Prismatic Crystals for schemes in characteristic $p$}
\author{Jiahong Yu ${}^*$}
\thanks{${}^*$Academy of Mathematics and Systems Science, Chinese Academy of Sciences} 
\thanks{${}^*$Morningside Center of Mathematics, CAS}
\def\Z{\mathbb{Z}}
\def\N{\mathbb{N}}

\def\Prism{\mathbbb{\Delta}}
\def\sq{\square}
\def\a{\mathfrak{a}}
\def\D{\mathbb{D}}
\def\M{\mathbb{M}}
\def\Ord{\mathbf{\Delta}}
\def\ep{\epsilon}
\def\C{\mathscr{C}}
\def\inte{\mathcal{O}}
\def\Spec{\mathrm{Spec}}
\def\X{\mathfrak{X}}
\def\dif{\mathrm{d}}
\def\Spec{\mathrm{Spec}}
\def\Spf{\mathrm{Spf}}
\def\DR{\mathrm{DR}}
\def\et{\mathrm{\text{\'et}}}
\def\shv{\mathrm{Shv}}
\def\nil{\text{$\mathrm{nil}$}}

\begin{document}

\pagestyle{fancy}
\fancyhf{}
\fancyhead[R]{\thepage}
\renewcommand{\headrulewidth}{0pt}

\begin{abstract}
{Let $(A,\delta_A)$ be a crystalline prism and let $\mathfrak X$ be a finite type $A/p$-scheme admitting a Koszul-regular closed immersion into a smooth formal $A$-scheme $Y$. We construct a sheaf of prismatic envelopes $\Prism_Y(\mathfrak X)$ attached to a Frobenius lift modulo $p^2$ on $Y$, prove that prismatic crystals on $(\mathfrak X/A)_{\Prism}$ are equivalent to integrable topologically quasi-nilpotent $p$-connections on $\Prism_Y(\mathfrak X)$, and identify their prismatic cohomology with the corresponding de Rham complex. When a global Frobenius lift is available, a lifted Ogus--Vologodsky functor gives an equivalence between $p$-connections on the prismatic envelope of the Frobenius twist and connections on the $p$-complete PD-envelope. Gluing this local correspondence yields an equivalence between prismatic crystals on $\mathfrak X^{(1)}$ and crystalline crystals on $\mathfrak X$ for l.c.i. $\mathfrak X$ over $A/p$.}
\end{abstract}

\maketitle

\setcounter{tocdepth}{1}
\tableofcontents

\section{Introduction}

We fix a prime $p$ and a crystalline prism $(A,\delta_A)$. Let $\mathfrak X$ be a quasi-compact quasi-separated finite type $A/p$-scheme. The aim of this paper is to describe prismatic crystals on $(\mathfrak X/A)_{\Prism}$ in terms of differential equations on an explicit prismatic envelope attached to an embedding of $\mathfrak X$ into a smooth formal $A$-scheme.
\subsection{Main results}
Assume that $\mathfrak X$ is embedded as a Koszul-regular closed subscheme of a smooth formal $A$-scheme $Y$, and that $Y$ is equipped with a Frobenius lift modulo $p^2$. We construct a sheaf of rings $\Prism_Y(\mathfrak X)$ on $\mathfrak X$ by gluing local prismatic envelopes. The resulting sheaf is independent of all local choices and satisfies the expected localization properties; see Theorem \ref{prismaticglobal}. Combining this construction with the local calculation of the \v{C}ech nerve of a prismatic envelope, which is ultimately reduced to toric coordinates in Theorems \ref{compare R} and \ref{compare D}, gives the main theorem of the paper.

\begin{theo}[= Theorem \ref{prismaticglobal} + \ref{globalD} + \ref{globalcohcompare}]\label{theo: intro main}
    Let $(A,\delta_A)$ be a crystalline prism, and let $\mathfrak X$ be a separated finite type $A/p$-scheme. Suppose that $\mathfrak X$ admits a Koszul-regular closed immersion into a smooth formal $A$-scheme $Y$, and suppose that $Y$ admits a Frobenius lift modulo $p^2$ compatible with the Frobenius on $A$. Then there is a canonically defined sheaf of rings $\Prism_Y(\mathfrak X)$ on $\mathfrak X$ equipped with a continuous $p$-differential
    $$p\dif:\Prism_Y(\mathfrak X)\to \Omega_{\Prism_Y(\mathfrak X)}^1.$$
    Moreover, there is a rank-preserving equivalence of categories
    $$\mathbf{Crys}_{(\mathfrak X/A)_{\Prism}}
    \xrightarrow{\sim}
    p\text{-}\mathbf{MIC}^{\nil}_{\Prism_Y(\mathfrak X)},$$
    where the right-hand side denotes the category of integrable topologically quasi-nilpotent $p$-connections on $\Prism_Y(\mathfrak X)$. If a prismatic crystal $\mathbb H$ corresponds to $(\mathcal M,\nabla)$ under this equivalence, then there is a canonical quasi-isomorphism
    $$R\nu_*\mathbb H\simeq \DR(\mathcal M,\nabla).$$
\end{theo}

The proof of Theorem \ref{theo: intro main} reduces to an affine local calculation. This local step yields the following key corollary, which is of independent interest as a lifted Ogus--Vologodsky correspondence:

\begin{theo}[= Theorem \ref{localOV}]
    Let $R$ be a $p$-complete smooth $A$-algebra with a Frobenius lift, and let $R\to B$ be a surjection whose kernel $\a$ is generated by $p$ and by lifts of a Koszul-regular sequence modulo $p$. Let $R^{(1)}:=A\widehat{\otimes}_{\phi_A,A}R$. Then there is a canonical rank-preserving equivalence of categories:
    $$C^{-1}:
    p\text{-}\mathbf{MIC}^{\nil}_{\Prism_{R^{(1)}}(\a R^{(1)})}
    \longrightarrow
    \mathbf{MIC}^{\nil}_{D_R(\a)}$$
    between the category of integrable topologically quasi-nilpotent $p$-connections on $\Prism_{R^{(1)}}(\a R^{(1)})$ and the category of integrable topologically quasi-nilpotent connections on $D_R(\a)$, where $D_R(\a)$ is the $p$-complete PD-envelope of $\a$ in $R$.
\end{theo}

By checking compatibility with restrictions to affine opens and with changes of coordinates, these local equivalences glue for l.c.i. schemes. This yields the following crystalline comparison theorem.

\begin{theo}[= Theorem \ref{equiv}]\label{intro: Ogus Volo}
    Let $\mathfrak X$ be an l.c.i. scheme over $A/p$, and let $\mathfrak X^{(1)}$ denote its Frobenius twist. Then there is a canonical equivalence of categories
    $$\mathbf{Crys}_{(\mathfrak X^{(1)}/A)_{\Prism}}
    \xrightarrow{\sim}
    \mathbf{Crys}_{(\mathfrak X/A)_{\mathrm{crys}}}.$$
\end{theo}

The proof of the main theorem rests on an elementary description of prismatic envelopes in characteristic $p$. If $R$ is a $p$-torsion-free $\delta$-ring and $I\subset R$ is a $\phi$-stable ideal containing $p$, then the prismatic envelope is the $p$-adic completion of the subring of $R[1/p]$ generated by $R$ and all iterated $\delta$-derivatives of $x/p$, for $x\in I$; see Definition \ref{def:prismatic_envelope}. The key point is that this construction depends on the $\delta$-structure only modulo $p$.

\begin{theo}[= Theorem \ref{key} and Corollary \ref{closure mod p}]\label{intro: indepnd mod p}
    Let $B$ be a $p$-torsion-free ring equipped with two $\delta$-structures $\delta_1$ and $\delta_2$ such that $\delta_1(a)\equiv \delta_2(a)\pmod p$ for all $a\in B$. Then, for every ideal $I\subset B$ containing $p$, the prismatic envelopes defined using $\delta_1$ and $\delta_2$ are canonically isomorphic:
    $$\Prism_{B,\delta_1}(I)\cong \Prism_{B,\delta_2}(I).$$
    More precisely, for every $x\in B$ and every $r\geq 0$, the subrings of $B[1/p]$ generated by
    $$\left\{\delta_1^t\left(\frac{x}{p}\right):0\leq t\leq r\right\}
    \quad\text{and}\quad
    \left\{\delta_2^t\left(\frac{x}{p}\right):0\leq t\leq r\right\}$$
    coincide.
\end{theo}

This independence result is the technical reason why the global construction of $\Prism_Y(\mathfrak X)$ only requires a Frobenius lift modulo $p^2$, rather than a global Frobenius lift on $Y$.
\subsection{Relation to previous work}

Our work builds naturally on recent foundational developments concerning the equivalence between prismatic crystals and $p$-connections, notably the work of Ogus \cite{ogus2023crystalline}, Tian \cite{Tian_2023}, and Wang \cite{wang2024prismatic}. Compared with these previous approaches, our results make the following improvements:

\begin{itemize}
    \item While previous global equivalences \cite{wang2024prismatic} strictly require the underlying scheme $\mathfrak{X}$ to be smooth, our framework applies to any finite-type scheme admitting a Koszul-regular closed immersion into a smooth formal scheme (e.g., syntomic schemes).
    
    \item We provide an elementary construction of prismatic envelopes in characteristic $p$ and prove that they depend on the $\delta$-structure only modulo $p$ (Theorem \ref{intro: indepnd mod p}). This algebraic independence removes the restrictive assumption of a global Frobenius lift and allows us to glue local data unambiguously to construct a global prismatic envelope sheaf $\Delta_Y(\mathfrak{X})$.
    
    \item Building on the global construction, we establish a natural equivalence between prismatic crystals on $\mathfrak{X}^{(1)}$ and crystalline crystals on $\mathfrak{X}$ (Theorem \ref{intro: Ogus Volo}). This generalizes a well-known observation of Bhatt and Scholze over perfect fields \cite[Example 4.7]{Bhatt_2023} to arbitrary syntomic schemes over arbitrary crystalline prisms $(A, (p))$, replacing topological triviality with a geometric correspondence via the Cartier transform.
\end{itemize}

\subsection*{Acknowledgments}

The author would like to thank Yupeng Wang for introducing this problem and for helpful discussions. The author also thanks Tian Qiu and Yupeng Wang for carefully reading a preliminary draft of this paper.

\section{Preliminaries on commutative algebra}

\subsection{Prismatic envelopes}\label{algebra}

We recall the construction of prismatic envelopes here. The construction in \cite{Bhatt_2022} is based on a homological construction and relies on some regularity assumptions. However, in characteristic $p$, we can construct prismatic envelopes in an elementary way without any regularity assumptions.

\begin{defi}\label{def:prismatic_envelope}
    Let $R$ be a $p$-torsion free $\delta$-ring and $I \subset R$ be a $\phi$-invariant ideal (which typically contains $p$). We define $\Prism^0_{R}(I)$ to be the smallest $\delta$-subring of $R\left[\frac{1}{p}\right]$ containing $R$ and $\frac{x}{p}$ for all $x \in I$. Explicitly,
    $$\Prism^0_{R}(I) = R\left[\delta^i\left(\frac{x}{p}\right): x\in I,\ i\geq 0\right].$$
    If $I$ is generated by $\{x_j\}_{j\in\Lambda}$, then $\Prism^0_{R}(I)$ is generated by $R$ and $\delta^i\left(\frac{x_j}{p}\right)$ for all $j \in \Lambda$ and $i \geq 0$.
\end{defi}

\begin{lemma}
    Let $R$ be a crystalline prism. For any ideal $I \subset R$ containing $p$, let $\Prism_{R}(I)$ be the $p$-adic completion of $\Prism^0_{R}(I)$. Then $\Prism_{R}(I)$ naturally admits the structure of a crystalline prism in $(A/R)_{\Prism}$ where $A=R/(p,I)$.
    
    Moreover, for any crystalline prism $(B, \delta_B)$, the $\delta$-morphisms from $\Prism_{R}(I)$ to $B$ map bijectively to the set of $\delta$-homomorphisms $R \to B$ sending $I$ into $pB$. We call $\Prism_{R}(I)$ the \emph{prismatic envelope} of $(R, I)$.
\end{lemma}

\begin{proof}
    By Definition \ref{def:prismatic_envelope}, we have $I \Prism^0_{R}(I) \subset p\Prism^0_{R}(I)$. Thus, there is a natural homomorphism $R/I \to \Prism^0_{R}(I)/p$, which uniquely extends to the $p$-adic completion $\Prism_{R}(I)$. This endows $\Prism_{R}(I)$ with the structure of a crystalline prism over $R$. 
    
    To verify the universal property, suppose we are given a crystalline prism $(B, \delta_B)$ and a $\delta$-homomorphism $f: R \to B$ such that $f(I) \subset pB$. Because $B$ is a $\delta$-ring and $p$-torsion free, the condition $f(x) \in pB$ for $x \in I$ implies that $f$ extends uniquely to a $\delta$-homomorphism on the subring $\Prism^0_{R}(I) \subset R[\frac{1}{p}]$ by sending $\frac{x}{p}$ to $\frac{f(x)}{p} \in B$. Since $B$ is $p$-complete, this map further extends uniquely to the $p$-adic completion $\Prism_{R}(I)$. This establishes the required bijection.
\end{proof}

\begin{rmk}
    The relative version holds canonically: if $R$ is an algebra over a crystalline prism $(A,\delta)$ compatible with $\phi$, then $\Prism_{R}(I)$ is a prism over $A$ satisfying the analogous universal property in $(R/A)_{\Prism}$.
\end{rmk}

\begin{cor}
    For a crystalline prism $(A,\delta)$ and an $A/p$-algebra $R$, the prismatic site $(R/A)_{\Prism}$ has arbitrary finite products.
\end{cor}

\begin{proof}
    Suppose $B$ and $C$ are two prisms in $(R/A)_{\Prism}$. Let $\bar{J} \subset B/p \otimes_{A/p} C/p$ be the ideal generated by $\{r\otimes 1 - 1\otimes r : r \in R\}$, and let $J$ be its preimage in the tensor product $B \otimes_A C$ (endowed with the tensor Frobenius lift). The canonical homomorphism $B \otimes_A C / (p, J) \to \Prism_{B\otimes_A C}(J)/p$ induces a well-defined map $R \to \Prism_{B\otimes_A C}(J)/p$. By the universal property of prismatic envelopes, $\Prism_{B\otimes_A C}(J)$ represents the coproduct of $B$ and $C$ in $(R/A)_{\Prism}$.
\end{proof}

Now we compare this construction with the constructions in \cite{Bhatt_2022}.
For the remainder of this subsection, we fix a $p$-torsion free $\delta$-ring $R$ and a sequence $x_1, \dots, x_r \in R$ whose image in $R/p$ is Koszul-regular. Let $I = (x_1, \dots, x_r)$.

\begin{prop}\label{prop:pd_envelope_koszul}
    Let $\Z[t_1, \dots, t_r] \to R$ be the map sending $t_j \mapsto x_j$, and let $D$ be the PD-envelope of $(t_1, t_2,\dots, t_r)$ in $\Z[t_1, t_2,\dots, t_r]$. Let \[D_R(I)=D\widehat{\otimes}_{\Z[t_1,t_2,\dots,t_r]}R.\] Then:
    \begin{enumerate}
        \item The PD-ring $D_R(I)$ is $p$-torsion free, and $R \otimes_{\Z[t_1, \dots, t_r]} D/p^n \xrightarrow{\sim} D_R(I)/p^n$ for all $n \geq 1$.
        \item Let $D^0 \subset R\left[\frac{1}{p}\right]$ be the subring generated by $R$ and all divided powers $\frac{x_j^m}{m!}$. Then, $D_R(I)$ is canonically isomorphic to the $p$-adic completion of $D^0$.
    \end{enumerate}
\end{prop}

\begin{proof}
    See \cite[Lemma 2.38]{Bhatt_2022}.
\end{proof}

\begin{prop}\label{prop:prism_close}
    Let $\tau: \Z\{u_1, \dots, u_r\} \to R$ be the $\delta$-homomorphism defined by $u_j \mapsto x_j$. Let $\Z\{v_1, \dots, v_r\}$ be a $\delta$-algebra over $\Z\{u_1, \dots, u_r\}$ via $u_j \mapsto p v_j$. Then there is a canonical isomorphism of prisms:
    $$\Prism_{R}(I) \cong R\widehat\otimes_{\tau, \Z\{u_1, \dots, u_r\}} \Z\{v_1, \dots, v_r\}.$$
\end{prop}

\begin{proof}
    By \cite[Example 7.9]{Bhatt_2022}, the right-hand side satisfies the same universal property as the prismatic envelope on the left hand side, hence they are canonically isomorphic.
\end{proof}

\begin{cor}\label{cor: crys comp}
    The natural map $R \widehat{\otimes}_{\phi, R} \Prism_{R}(I) \xrightarrow{id \otimes \phi} \Prism_{R}\big((\phi(x_1), \dots, \phi(x_r))\big)$ is an isomorphism. Furthermore, there is a canonical isomorphism:
    $$R \otimes_{\phi, R} \Prism_{R}(I) \cong D_R(I).$$
\end{cor}

\begin{proof}
    Since $\phi(x_j) \equiv x_j^p \pmod p$, the sequence $\phi(x_1), \dots, \phi(x_r)$ remains Koszul-regular modulo $p$ (cf. \cite[\href{https://stacks.math.columbia.edu/tag/062G}{Tag 062G}]{stacks-project}). The first isomorphism follows directly from Proposition \ref{prop:prism_close} by extending scalars along $\phi: R \to R$ (which corresponds to changing the base map from $\tau_1$ to $\tau_2$). 
    
    For the second isomorphism, see \cite[Corollary 2.39]{Bhatt_2022}.
\end{proof}

\subsection[Independence of delta-structure modulo p]{Independence of $\delta$-structure modulo $p$}

Our key observation is that the prismatic envelope depends on the $\delta$-structure only modulo $p$. More explicitly, fix a $p$-torsion free ring $B$ equipped with two $\delta$-structures $\delta_1$ and $\delta_2$ such that
$$\delta_1(x) - \delta_2(x) \in pB \quad \forall x \in B.$$
Recall that they induce two Frobenius lifts $\phi_i(x) := x^p + p\delta_i(x)$. We extend $\phi_i$ to $B[\frac{1}{p}]$ by localization, which uniquely extends the $\delta$-structures to $B[\frac{1}{p}]$ via $\delta_i(x) = \frac{\phi_i(x)-x^p}{p}$ for $i \in \{1,2\}$. We will prove the following theorem:

\begin{theo}\label{key}
    Let $B$ be a $p$-torsion free ring with two $\delta$-structures satisfying $\delta_1(a) - \delta_2(a) \in pB$ for all $a \in B$. For any $x \in B$ and integer $r \geq 0$, define the subrings of $B[\frac{1}{p}]$:
    $$B_i^r = B\left[\delta_i^t\left(\frac{x}{p}\right) : 0 \leq t \leq r\right]$$ 
    for $i \in \{1,2\}$. Then as subrings of $B\left[\frac{1}{p}\right]$, we have $B_1^r = B_2^r$ for all $r \geq 0$.
\end{theo}

\begin{lemma}\label{diff}
    Let $A$ be a $\delta$-ring (not necessarily $p$-torsion free) with Frobenius lift $\phi(x) = x^p + p\delta(x)$, and let $n \geq 1$ be an integer. Suppose $d: A \to A$ is an additive endomorphism such that for all $x,y \in A$:
    $$d(xy) = \phi(x)d(y) + \phi(y)d(x) + p^n d(x)d(y).$$
    Then:
    \begin{enumerate}
        \item[(1)] For any $l \geq 1$ and $x \in A$, $d(x^l) = \sum_{i=0}^{l-1}\binom{l}{i}p^{n(l-i-1)}\phi(x)^i d(x)^{l-i}$.
        \item[(2)] $d\left(\phi^r(x)\right) \in p^rA$ for any $x \in A$ and $r \geq 0$.
    \end{enumerate}
\end{lemma}

\begin{proof}
    (1) We proceed by induction on $l$. The base case $l=1$ is trivial. Assume the claim holds for $l-1$. By the induction hypothesis, we have:
    \begin{align*}
        d(x^l) &= d(x \cdot x^{l-1}) = \phi(x)d(x^{l-1}) + \phi(x)^{l-1}d(x) + p^n d(x)d(x^{l-1}) \\
        &= \phi(x)\left(\sum_{i=0}^{l-2}\binom{l-1}{i}p^{n(l-2-i)}\phi(x)^i d(x)^{l-1-i}\right) + \phi(x)^{l-1}d(x) \\
        &\quad + p^n d(x)\sum_{j=0}^{l-2}\binom{l-1}{j}p^{n(l-2-j)}\phi(x)^j d(x)^{l-1-j} \\
        &= \sum_{i=0}^{l-1}\binom{l-1}{i-1}p^{n(l-1-i)}\phi(x)^i d(x)^{l-i} + \binom{l-1}{l-1}\phi(x)^{l-1}d(x) \\
        &\quad + \sum_{j=0}^{l-2}\binom{l-1}{j}p^{n(l-1-j)}\phi(x)^j d(x)^{l-j} \\
        &= \sum_{i=0}^{l-1}\left[\binom{l-1}{i-1} + \binom{l-1}{i}\right]p^{n(l-1-i)}\phi(x)^i d(x)^{l-i} \\
        &= \sum_{i=0}^{l-1}\binom{l}{i}p^{n(l-1-i)}\phi(x)^i d(x)^{l-i}.
    \end{align*}
    This completes the inductive step for $l$.

    (2) We proceed by induction on $r$, with $r=0$ being trivial. Assume the claim holds for all $r < r_0$ where $r_0 > 0$. For $r=r_0$ and any $x \in A$:
    \begin{align*}
        d\left(\phi^{r_0}(x)\right) &= d\left(\phi^{r_0-1}\left(\phi(x)\right)\right) = d\left(\phi^{r_0-1}\left(x^p + p\delta(x)\right)\right) \\
        &= d\left(\phi^{r_0-1}(x)^p\right) + p d\left(\phi^{r_0-1}\left(\delta(x)\right)\right).
    \end{align*}
    Since $\delta(x) \in A$, the term $p d\left(\phi^{r_0-1}\left(\delta(x)\right)\right) \in p^{r_0}A$ by the induction hypothesis. It remains to show $d\left(\phi^{r_0-1}(x)^p\right) \in p^{r_0}A$. By part (1), we have:
    \begin{equation*}\label{*}
        d\left(\phi^{r_0-1}(x)^p\right) = \sum_{i=0}^{p-1}\binom{p}{i}p^{n(p-1-i)}\phi\left(\phi^{r_0-1}(x)\right)^i d\left(\phi^{r_0-1}(x)\right)^{p-i}. \tag{*}
    \end{equation*}
    By the induction hypothesis, $d\left(\phi^{r_0-1}(x)\right)^{p-i} \in p^{(r_0-1)(p-i)}A$. When $i=0$, the corresponding term lies in $p^{p(r_0-1)+n(p-1)}A \subset p^{r_0}A$. When $0 < i < p$, we have $p \mid \binom{p}{i}$, and hence the term lies in $p^{(p-i)(r_0-1)+1}A \subset p^{r_0}A$. Thus, each term in the summation \eqref{*} lies in $p^{r_0}A$, completing the proof.
\end{proof}

Returning to the proof of Theorem \ref{key}, recall that $\delta_1 \equiv \delta_2 \pmod p$ on $B$. For any $x \in B$ and integers $1 \leq l \leq r$, define:
$$\lambda(r,l) = \frac{1}{p^l}\left[\phi_1^{l}\left(\delta_1^{r-l}\left(\frac{x}{p}\right)\right) - \phi_2\circ\phi_1^{l-1}\left(\delta_1^{r-l}\left(\frac{x}{p}\right)\right)\right] \in B\left[\frac{1}{p}\right].$$

\begin{lemma}\label{main cal}
    \begin{enumerate}
        \item[(1)] For any $1 \leq l \leq r$:
        $$\lambda(r+1,l) = \lambda(r+1,l+1) - \frac{1}{p^{l+1}}\left[\phi_1^{l}\left(\delta_1^{r-l}\left(\frac{x}{p}\right)\right)^p - \phi_2\circ\phi_1^{l-1}\left(\delta_1^{r-l}\left(\frac{x}{p}\right)\right)^p\right].$$
        \item[(2)] $\lambda(r,l) \in B_1^{r-1}$ for all $1 \leq l \leq r$.
    \end{enumerate}
\end{lemma}

\begin{proof}
    (1) Observe that:
    \begin{align*}
        \lambda(r+1,l) &= \frac{1}{p^{l}}\left[\phi_1^{l}\left(\delta_1\left(\delta_1^{r-l}\left(\frac{x}{p}\right)\right)\right) - \phi_2\circ\phi_1^{l-1}\left(\delta_1\left(\delta_1^{r-l}\left(\frac{x}{p}\right)\right)\right)\right]\\ 
        &= \frac{1}{p^{l}}\left[\phi_1^{l}\left(\frac{\phi_1\left(\delta_1^{r-l}\left(\frac{x}{p}\right)\right) - \delta_1^{r-l}\left(\frac{x}{p}\right)^p}{p}\right) - \phi_2\circ\phi_1^{l-1}\left(\frac{\phi_1\left(\delta_1^{r-l}\left(\frac{x}{p}\right)\right) - \delta_1^{r-l}\left(\frac{x}{p}\right)^p}{p}\right)\right]\\
        &= \frac{1}{p^{l+1}}\bigg[\phi_1^{l+1}\left(\delta_1^{r-l}\left(\frac{x}{p}\right)\right) - \phi_1^{l}\left(\delta_1^{r-l}\left(\frac{x}{p}\right)\right)^p \\
        &\quad - \phi_2\circ\phi_1^{l}\left(\delta_1^{r-l}\left(\frac{x}{p}\right)\right) + \phi_2\circ\phi_1^{l-1}\left(\delta_1^{r-l}\left(\frac{x}{p}\right)\right)^p \bigg]\\
        &= \lambda(r+1,l+1) - \frac{1}{p^{l+1}}\left[\phi_1^{l}\left(\delta_1^{r-l}\left(\frac{x}{p}\right)\right)^p - \phi_2\circ\phi_1^{l-1}\left(\delta_1^{r-l}\left(\frac{x}{p}\right)\right)^p\right].
    \end{align*}

    We now prove part~(2) by induction on $r$. First note that \[ \varphi_1(B_1^m)\subseteq B_1^{m+1} \qquad \text{for every } m\geq 0. \] Indeed, $\varphi_1(B)\subseteq B$, and for every $0\leq t\leq m$, \[ \varphi_1 \left( \delta_1^t\left(\frac{x}{p}\right) \right) = \delta_1^t\left(\frac{x}{p}\right)^p + p\, \delta_1^{t+1}\left(\frac{x}{p}\right) \in B_1^{m+1}. \] Consequently, \[ \varphi_1^l(B_1^m)\subseteq B_1^{m+l} \qquad \text{for all } l,m\geq 0. \tag{*} \] 
    
    For $r=1$, we necessarily have $l=1$. Since $\varphi_i(x)=x^p+p\delta_i(x)$, we obtain \[  \lambda(1,1) = \frac{1}{p} \left[ \varphi_1\left(\frac{x}{p}\right) - \varphi_2\left(\frac{x}{p}\right) \right] = \frac{\delta_1(x)-\delta_2(x)}{p} \in B = B_1^0. \] Thus, the assertion holds when $r=1$. Assume now that $r\geq 2$ and that the assertion has been proved for $r-1$. We first show that \[ \lambda(r,r)\in B. \] Since $\delta_1\equiv\delta_2\pmod p$, the map \[ d := \frac{\varphi_2-\varphi_1}{p^2} : B\longrightarrow B \] is well-defined and additive. Moreover, the multiplicativity of $\varphi_1$ and $\varphi_2$ gives \[ d(ab) = \varphi_1(a)d(b) + \varphi_1(b)d(a) + p^2d(a)d(b). \] Applying Lemma~\ref{diff} with $n=2$, we obtain \[ d\left(\varphi_1^{r-1}(x)\right)\in p^{r-1}B. \] Therefore, \[ \lambda(r,r) = \frac{1}{p^r} \left[ \varphi_1^r\left(\frac{x}{p}\right) - \varphi_2\circ\varphi_1^{r-1} \left(\frac{x}{p}\right) \right] = - \frac{ d\left(\varphi_1^{r-1}(x)\right) }{ p^{r-1} } \in B \subseteq B_1^{r-1}. \tag{**} \] It remains to prove that \[ \lambda(r,l)\in B_1^{r-1} \qquad \text{for } 1\leq l\leq r-1. \] We argue by descending induction on $l$. Fix such an $l$ and assume that \[ \lambda(r,l+1)\in B_1^{r-1}. \] Set \[ z := \delta_1^{r-1-l}\left(\frac{x}{p}\right), \qquad a:=\varphi_1^l(z), \qquad b:=\varphi_2\circ\varphi_1^{l-1}(z). \] By the induction hypothesis on $r$, we have \[ c := \lambda(r-1,l) = \frac{a-b}{p^l} \in B_1^{r-2}. \] Thus, \[ a=b+p^lc. \tag{***} \] Since \[ z\in B_1^{r-1-l}, \] formula~\textup{(*)} implies that \[ a=\varphi_1^l(z)\in B_1^{r-1}. \] Moreover, by~\textup{(***)}, \[ b=a-p^lc\in B_1^{r-1}. \] We now expand: \[ \begin{aligned} \frac{a^p-b^p}{p^{l+1}} &= \frac{(b+p^lc)^p-b^p}{p^{l+1}} \\ &= \sum_{j=1}^{p-1} \frac{1}{p} \binom{p}{j} p^{l(j-1)} b^{p-j}c^j + p^{l(p-1)-1}c^p. \end{aligned} \] For $1\leq j\leq p-1$, we have \[ \frac{1}{p}\binom{p}{j}\in\mathbf{Z}. \] Furthermore, \[ l(p-1)-1\geq 0 \] because $l\geq 1$ and $p\geq 2$. Since $b,c\in B_1^{r-1}$, it follows that \[ \frac{a^p-b^p}{p^{l+1}} \in B_1^{r-1}. \] Part~(1) now yields \[ \lambda(r,l) = \lambda(r,l+1) - \frac{a^p-b^p}{p^{l+1}} \in B_1^{r-1}. \] Starting from~\textup{(**)} and descending from $l=r-1$ to $l=1$, we conclude that \[ \lambda(r,l)\in B_1^{r-1} \qquad \text{for every } 1\leq l\leq r. \] This completes the induction.
\end{proof}

\begin{proof}[Proof of Theorem \ref{key}]
    We proceed by induction on $r$. The base case $r=0$ is trivial. Assume the claim holds for $r-1$. For the inductive step $r$, by symmetry, it suffices to prove $\delta^r_1\left(\frac{x}{p}\right) \in B_2^r$. By the construction of $B_i^r$, we have $\delta_i(B_i^l) \subset B_i^{l+1}$ for $i \in \{1,2\}$ and $l \geq 0$. Thus, by the induction hypothesis, $\delta_1^{r-1}\left(\frac{x}{p}\right) \in B_2^{r-1}$, whence $\delta_2\left(\delta_1^{r-1}\left(\frac{x}{p}\right)\right) \in B_2^r$. 
    
    Note that for any $t$:
    $$\delta_1^r(t) - \delta_2\left(\delta_1^{r-1}(t)\right) = \frac{1}{p}\left[\phi_1\left(\delta_1^{r-1}(t)\right) - \phi_2\left(\delta_1^{r-1}(t)\right)\right].$$
    Hence, taking $t = x/p$, we obtain:
    $$\delta^r_1\left(\frac{x}{p}\right) = \delta_2\circ\delta^{r-1}_1\left(\frac{x}{p}\right) + \lambda(r,1) \in B_2^r$$ 
    by Lemma \ref{main cal}(2).
\end{proof}

\begin{cor}\label{closure mod p}
    Let $B$ be a $p$-torsion free ring equipped with two $\delta$-structures $\delta_1$ and $\delta_2$ such that $\delta_1 \equiv \delta_2 \pmod p$. For any ideal $I \subset B$ containing $p$, the prismatic envelopes $\Prism_{B,\delta_1}(I)$ and $\Prism_{B,\delta_2}(I)$ are canonically isomorphic.
\end{cor}

\begin{proof}
    By Definition \ref{def:prismatic_envelope}, the uncompleted prismatic envelope $\Prism^0_{B,\delta_i}(I)$ is the subring of $B\left[\frac{1}{p}\right]$ generated by $B$ and all elements of the form $\delta_i^t\left(\frac{x}{p}\right)$ for $x \in I$ and $t \geq 0$. 
    
    Applying Theorem \ref{key} and taking the union over all $r \geq 0$, we see that for any fixed $x \in I$, the subring generated by $\{\delta_1^t\left(\frac{x}{p}\right)\}_{t \geq 0}$ coincides exactly with the subring generated by $\{\delta_2^t\left(\frac{x}{p}\right)\}_{t \geq 0}$ inside $B\left[\frac{1}{p}\right]$. 
    
    Since $\Prism^0_{B,\delta_i}(I)$ is the compositum of these subrings for all generators $x \in I$, it follows that
    $$\Prism^0_{B,\delta_1}(I) = \Prism^0_{B,\delta_2}(I)$$
    as identical subrings of $B\left[\frac{1}{p}\right]$. Passing to the $p$-adic completions on both sides yields the canonical isomorphism $\Prism_{B,\delta_1}(I) \cong \Prism_{B,\delta_2}(I)$.
\end{proof}

\subsection{Some calculations on toric charts}\label{toric}

In this subsection, we fix a crystalline prism $(A,\delta)$, a $p$-completely smooth $A$-algebra $R$, and a $p$-completely \'etale homomorphism
$$\sq:A\langle X_1^{\pm 1},X_2^{\pm 1},\dots,X_d^{\pm 1}\rangle\to R.$$ 
For any integer $m \geq 0$ and any $p$-torsion free, $p$-complete ring $S$, we denote $S/p^{m+1}$ by $S_m$. Fix an ideal $\a \subset R$ containing $p$, and let $B = R/\a$. Suppose $R$ is equipped with two $\delta$-structures $\delta_1$ and $\delta_2$, with corresponding Frobenius lifts $\phi_1$ and $\phi_2$. Let
$$P_i = \Prism_{R,\delta_i}\left(\a\right)$$ 
for $i \in \{1,2\}$. Then the natural diagram 
$$B \to P_i/p \leftarrow P_i$$ 
forms an object in $(B/A)_{\Prism}$. Let $J$ be the kernel of the diagonal map $R \widehat{\otimes}_A R \to R$.

\begin{lemma}
    The coproduct of $(B \to P_1/p \leftarrow P_1)$ and $(B \to P_2/p \leftarrow P_2)$ in $(B/A)_{\Prism}$ is represented by the prismatic envelope
    $$P_{1,2} := \Prism_{R\widehat{\otimes}_A R}\left((x\otimes 1: x \in \a) + J\right).$$
\end{lemma}

\begin{proof}
    This follows directly from the universal property of prismatic envelopes.
\end{proof}

\begin{defi}
    Let $R_{1,2}$ denote the prismatic envelope $\Prism_{R\widehat{\otimes}_A R}\left((p,J)\right)$.
\end{defi}

\begin{cor}\label{RtoD}
    There is a canonical isomorphism $P_{1,2} \cong \Prism_{R_{1,2}}\left((x\otimes 1: x \in \a)\right)$.
\end{cor}

From now on, we assume $\delta_1 \equiv \delta_2 \pmod p$. We will use Corollary \ref{RtoD} to explicitly compute the ring $P_{1,2}$. The following structural result was first established by Wang in \cite{wang2024prismatic}. Here, we provide an alternative proof using Corollary \ref{closure mod p}.

\begin{theo}\label{compare R}
    There is a canonical isomorphism 
    $$\iota_\square: R_{1,2} \cong R\left[Y_1, Y_2, \dots, Y_d\right]_{\mathrm{pd}}^\wedge,$$ 
    where the right-hand side is the $p$-adic completion of the PD-polynomial ring. Under this isomorphism, the homomorphism $R \xrightarrow{\text{id}\otimes 1} R_{1,2}$ is identified with the standard scalar embedding, and $R \xrightarrow{1\otimes \text{id}} R_{1,2}$ is identified with the unique embedding $R \to R\left[Y_1, \dots, Y_d\right]_{\mathrm{pd}}^\wedge$ sending $X_i$ to $X_i + pY_i$.
\end{theo}

To prove this theorem, we need a lemma analogous to \cite[Lemma 2.28(1)]{ogus2023crystalline}. The inductive structure of our proof is inspired by Wang's arguments in \cite{wang2024prismatic}.

\begin{lemma}\label{lem:ogus_analog}
    Suppose $(A,\delta)$ is a $p$-torsion free $\delta$-ring and $I \subset A$ is an ideal satisfying:
    \begin{enumerate}
        \item[(1)] $\phi(I) \subset I$;
        \item[(2)] $A/I$ is $p$-torsion free.
    \end{enumerate}
    Then $\Prism_A^0\left((p,I)\right)$ coincides with the $A$-subalgebra of $A\left[\frac{1}{p}\right]$ generated by 
    $$\left\{\frac{x^{p^m}}{p^{1+p+\dots +p^m}} : x \in I,\ m \geq 0\right\}.$$
\end{lemma}

\begin{proof}
    Let $A_1^m$ and $A_2^m$ be the $A$-subalgebras of $A\left[\frac{1}{p}\right]$ defined by:
    $$A_1^m := A\left[\frac{x^{p^i}}{p^{1+p+\dots +p^i}} : x \in I,\ 0 \leq i \leq m\right], \quad A_2^m := A\left[\delta^i\left(\frac{x}{p}\right) : 0 \leq i \leq m\right].$$
    We will prove by induction on $m$ that $A_1^m = A_2^m$. Furthermore, we claim that for each $m$, there exists a sign $\ep \in \{\pm 1\}$ such that 
    $$\frac{x^{p^m}}{p^{1+p+\dots +p^m}} - \ep\delta^m\left(\frac{x}{p}\right) \in A_1^{m-1}.$$
    The base case $m=0$ is trivial. Assume the claim holds for all integers strictly less than $m$. It suffices to show:
    $$\frac{x^{p^m}}{p^{1+p+\dots +p^m}} - \ep\delta^m\left(\frac{x}{p}\right) \in A_1^{m-1} (= A_2^{m-1})$$ 
    for some $\ep \in \{\pm 1\}$. By the induction hypothesis, there exist $z \in A_1^{m-2} = A_2^{m-2}$ and $\ep' \in \{\pm 1\}$ such that 
    $$\delta^{m-1}\left(\frac{x}{p}\right) = \ep'\frac{x^{p^{m-1}}}{p^{1+p+\dots +p^{m-1}}} + z.$$
    Applying $\delta$, we obtain:
    \begin{align*}
        \delta^m\left(\frac{x}{p}\right) &= \delta\left(\ep'\frac{x^{p^{m-1}}}{p^{1+p+\dots +p^{m-1}}} + z\right) \\
        &= \delta\left(\ep'\frac{x^{p^{m-1}}}{p^{1+p+\dots +p^{m-1}}}\right) + \delta(z) - \sum_{j=1}^{p-1}\frac{1}{p}\binom{p}{j}\left(\ep'\frac{x^{p^{m-1}}}{p^{1+p+\dots +p^{m-1}}}\right)^j z^{p-j}.
    \end{align*}
    Since $z \in A_2^{m-2}$, we have $\delta(z) \in A_2^{m-1} = A_1^{m-1}$. Moreover, powers of $\frac{x^{p^{m-1}}}{p^{1+p+\dots +p^{m-1}}}$ belong to $A_1^{m-1}$. Thus, we only need to show:
    $$\delta\left(\ep'\frac{x^{p^{m-1}}}{p^{1+p+\dots +p^{m-1}}}\right) - \ep\frac{x^{p^m}}{p^{1+p+\dots +p^m}} \in A_1^{m-1} = A_2^{m-1}$$ 
    for some $\ep \in \{\pm 1\}$. The assumption $\phi(I) \subset I$ implies $\delta(I) \subset I$, allowing us to expand:
    \begin{align*}
        \delta\left(\ep'\frac{x^{p^{m-1}}}{p^{1+p+\dots +p^{m-1}}}\right) &= \frac{1}{p}\left[ \ep'\phi\left(\frac{x^{p^{m-1}}}{p^{1+p+\dots +p^{m-1}}}\right) - (\ep')^p\left(\frac{x^{p^{m-1}}}{p^{1+p+\dots +p^{m-1}}}\right)^p \right] \\
        &= \ep'\frac{\phi(x)^{p^{m-1}}}{p^{1+1+p+\dots +p^{m-1}}} + (-(\ep')^p)\frac{x^{p^m}}{p^{1+p+\dots +p^m}}.
    \end{align*}
    Recall that $\phi(x) = x^p + p\delta(x) = p^2\left(p^{p-2}\left(\frac{x}{p}\right)^p + \frac{\delta(x)}{p}\right) \in p^2 A_1^0 = p^2 A_2^0$. Consequently:
    $$\frac{\phi(x)^{p^{m-1}}}{p^{1+1+p+\dots +p^{m-1}}} \in p^{2p^{m-1} - (1+1+p+\dots +p^{m-1})}A_1^0 \subset A_1^0.$$ 
    Setting $\ep = -(\ep')^p$ completes the inductive step.
\end{proof}

\begin{proof}[Proof of Theorem \ref{compare R}]
    By Corollary \ref{closure mod p}, we may assume without loss of generality that $\delta_1 = \delta_2$. Then $J$ is a $\phi$-invariant ideal in the $p$-torsion free quotient $R \widehat{\otimes}_A R$. By Lemma \ref{lem:ogus_analog} and the definition of the prismatic envelope, we have:
    $$R_{1,2} = \left( R\widehat{\otimes}_A R\left[\frac{x^{p^m}}{p^{1+p+\dots +p^m}} : x \in J\right] \right)^\wedge.$$
    
    Consider the continuous homomorphism $R\left[Y_1, \dots, Y_d\right]_{\mathrm{pd}}^\wedge \to R_{1,2}$ defined by $a \mapsto a \otimes 1$ for $a \in R$ and $Y_i \mapsto \frac{1\otimes X_i - X_i\otimes 1}{p}$. To show this is an isomorphism, it suffices by completeness to verify it modulo $p$, i.e., that the map
    \begin{equation}\label{con:*}
        R\left[Y_1, \dots, Y_d\right]_{\mathrm{pd}} \to R\widehat{\otimes}_A R\left[\frac{x^{p^m}}{p^{1+p+\dots +p^m}} : x \in J\right]
    \end{equation}
    induces an isomorphism modulo $p$. 
    
    The proof of this reduction essentially follows Wang \cite{wang2024prismatic}. The key observation is the isomorphism modulo $p$:
    \begin{equation}\label{con:**}
        R\left[Y_1, \dots, Y_d\right]/p \cong R\widehat{\otimes}_A R\left[\frac{x}{p} : x \in J\right]/p,
    \end{equation}
    where $Y_i \mapsto \frac{1\otimes X_i - X_i\otimes 1}{p}$. This follows straightforwardly from the \'etaleness of $\square$. Under the isomorphism \eqref{con:**}, the ideal $(Y_1, \dots, Y_d)$ on the left corresponds precisely to the ideal generated by $\{\frac{x}{p} : x \in J\}$ on the right. Taking the PD-envelopes of these corresponding ideals yields the desired isomorphism \eqref{con:*}.
\end{proof}

\begin{theo}\label{compare D}
    Assume $\delta_1 \equiv \delta_2 \pmod p$. There is a canonical isomorphism
    $$P_{1,2} \cong P_1\left[Y_1, \dots, Y_d\right]_{\mathrm{pd}}^\wedge.$$
\end{theo}

\begin{proof}
    By Corollary \ref{closure mod p}, we may again assume $\delta_1 = \delta_2$. Then, combining Corollary \ref{RtoD} and Theorem \ref{compare R}, we obtain:
    $$P_{1,2} \cong \Prism_{R_{1,2}}\left(\a R_{1,2}\right) \cong \Prism_{R\left[Y_1, \dots, Y_d\right]_{\mathrm{pd}}^\wedge}\left(\a R\left[Y_1, \dots, Y_d\right]_{\mathrm{pd}}^\wedge\right) \cong P_1\left[Y_1, \dots, Y_d\right]_{\mathrm{pd}}^\wedge.$$
    Here, the final isomorphism follows from the flat base change of prismatic envelopes.
\end{proof}

\section{Local construction of the functor \texorpdfstring{$\D_{\Prism}$}{D-Prism}}\label{Dprism}

We maintain the notation from Subsection \ref{toric}, but now fix a single $\delta$-structure $\delta$ on $R$. Let $P = \Prism_{R}\left(\a\right)$ and $\widehat{\Omega}_P := P \otimes_R \widehat{\Omega}_{R/A}$.

In this section, we compare prismatic crystals on the site $(B/A)_{\Prism}$ with a specific class of $p$-connections over $P$. While this comparison was originally established by Ogus in \cite{ogus2023crystalline}, we provide a more elementary proof here. Note that the case where $\a = 0$ was handled by Wang \cite{wang2024prismatic}, and a characteristic $p$ analogue for $\a = 0$ was proven by Yichao Tian \cite{Tian_2023}. The following foundational lemma was first discovered by Ogus \cite[Proposition 3.6]{ogus2023crystalline}; we prove a slightly generalized version.

\begin{lemma}\label{lem: extend dif}
    The differential operator $p\dif: R \to \widehat{\Omega}_{R/A}$ (the $p$-th multiple of the standard exterior derivative) extends uniquely to a $p$-adically continuous derivation $p\dif: P \to \widehat{\Omega}_P$.
\end{lemma}

\begin{proof}
   Uniqueness follows immediately from the linearity of $p\dif$. Let $P^0 = \Delta_R^0(\mathfrak{a})$. We first extend $p\dif$ to $R\left[\frac{1}{p}\right]$ and show that $p\dif(P^0) \subset P^0 \otimes_R \widehat{\Omega}_{R/A}$.

We proceed by lexicographic induction on the pair $(n, j)$, where $n = i+j$. We claim that for any $x \in \mathfrak{a}$,
\[
p^{-i+1}\dif\left(\phi^i\left(\delta^j\left(\frac{x}{p}\right)\right)\right) \in P^0 \otimes_R \widehat{\Omega}_{R/A}.
\]

For $j=0$, the claim holds by a straightforward secondary induction on $i$.

Assume the claim holds for all pairs $(i', j')$ such that either $i' + j' \le \lambda$, or $i' + j' = \lambda + 1$ with $j' < j$. Now consider a pair $(i, j)$ where $i+j = \lambda+1$ and $j>0$. Using the identity $\delta(t) = \frac{\phi(t) - t^p}{p}$ and the fact that $\phi$ is a ring homomorphism, we compute:
\begin{align*}
p^{-i+1}\dif\left(\phi^i\left(\delta^j\left(\frac{x}{p}\right)\right)\right) &= p^{-i+1}\dif\left(\phi^i\left(\frac{\phi\left(\delta^{j-1}\left(\frac{x}{p}\right)\right) - \delta^{j-1}\left(\frac{x}{p}\right)^p}{p}\right)\right) \\
&= p^{-i}\dif\left(\phi^{i+1}\left(\delta^{j-1}\left(\frac{x}{p}\right)\right)\right) - p^{-i+1}\phi^i\left(\delta^{j-1}\left(\frac{x}{p}\right)\right)^{p-1}\dif\left(\phi^i\left(\delta^{j-1}\left(\frac{x}{p}\right)\right)\right).
\end{align*}

The first term belongs to $P^0 \otimes_R \widehat{\Omega}_{R/A}$ by the inductive hypothesis for $(i+1, j-1)$, since $(i+1)+(j-1) = \lambda+1$ and $j-1 < j$. 

The second term also belongs to $P^0 \otimes_R \widehat{\Omega}_{R/A}$ because $\phi^i\left(\delta^{j-1}\left(\frac{x}{p}\right)\right)^{p-1} \in P^0$, and the remaining factor $p^{-i+1}\dif\left(\phi^i\left(\delta^{j-1}\left(\frac{x}{p}\right)\right)\right)$ satisfies the inductive hypothesis for $(i, j-1)$, where $i + (j-1) = \lambda < \lambda+1$.

In particular, taking $i=0$ yields $p\dif\left(\delta^j\left(\frac{x}{p}\right)\right) \in P^0 \otimes_R \widehat{\Omega}_{R/A}$, completing the proof.
\end{proof}

\begin{cor}\label{higherpdif}
    The higher differential operator $p\dif: \widehat{\Omega}^k_{R/A} \to \widehat{\Omega}^{k+1}_{R/A}$ extends uniquely to a $p$-adically continuous map
    $$P \otimes_R \widehat{\Omega}_{R/A}^k \to P \otimes_R \widehat{\Omega}_{R/A}^{k+1}.$$
\end{cor}
\begin{proof}
    By Lemma \ref{lem: extend dif}, the operator $p\mathrm{d}$ extends uniquely to a continuous derivation $p\mathrm{d}: P \to \widehat{\Omega}_P$. We can explicitly define its extension to the $k$-th exterior power $P \otimes_R \widehat{\Omega}^k_{R/A}$ via the Leibniz rule:
\[
p\mathrm{d}(f \otimes \omega) = p\mathrm{d}(f) \wedge \omega + f \otimes p\mathrm{d}(\omega)
\]
for any $f \in P$ and $\omega \in \widehat{\Omega}^k_{R/A}$. This map is well-defined over $R$ precisely because $p\mathrm{d}$ acts as a derivation on both $P$ and $R$. The $p$-adic continuity and uniqueness of this higher extension follow immediately from the continuity and uniqueness of the base derivation on $P$.
\end{proof}

Let $p\partial_i$ denote the extension of $p\frac{\partial}{\partial X_i}$ to $P$. Recall that under the isomorphism of Theorem \ref{compare R}, the extension of the left inclusion $\text{id}\otimes 1: R \to R \widehat{\otimes}_A R$ to $P \to P_{1,2}$ is identified with the standard scalar inclusion $P \to P\left[Y_1, \dots, Y_d\right]_{\mathrm{pd}}^\wedge$.

\begin{prop}
    The extension of the right inclusion $1\otimes \text{id}: R \to R \widehat{\otimes}_A R$ to $P \to P_{1,2}$ is identified with the Taylor expansion map:
    $$f \mapsto \sum_{\alpha \in \N^d} p^{|\alpha|}\partial_{\alpha}(f)\underline{Y}^{[\alpha]}.$$
\end{prop}

\begin{proof}
    Consider the following commutative diagram (with solid arrows):
   \begin{center} $$
\begin{tikzcd}
A\langle X_1^{\pm 1}, \dots, X_d^{\pm 1} \rangle \arrow[r, "1 \otimes id"] \arrow[d, "\Box"'] & R[Y_1, \dots, Y_d]_{pd}^\wedge \arrow[d, "Y_i \mapsto 0"] \\
R \arrow[r, "id"'] \arrow[ru, dashed, "\exists !"] & R
\end{tikzcd}
$$\end{center}
    By the \'etaleness of $\sq$, there is a unique dashed arrow lifting the identity on $R$ that makes the diagram commute. Since $X_i \otimes 1 + pY_i = 1 \otimes X_i$, the continuous map $f \mapsto \sum_{\alpha \in \N^d} p^{|\alpha|}\partial_{\alpha}(f)\underline{Y}^{[\alpha]}$ provides exactly such a lift. By uniqueness, the claim follows.
\end{proof}

\subsection{Conventions for cosimplicial rings and stratifications}

Our construction proceeds in two steps. First, for each prismatic crystal, we construct a stratification on the \v{C}ech nerve of $P$ (yielding the functor $\M_{\Prism}$). Subsequently, we construct a $p$-connection for each such stratification.

For clarity, we briefly review our conventions regarding cosimplicial rings. Let $\Ord$ denote the category of finite non-empty totally ordered sets, and let $[n]$ denote the ordered set $\{0,1,\dots,n\}$. Recall that a cosimplicial ring is a covariant functor from $\Ord$ to the category of commutative rings. We will typically denote a cosimplicial ring by $S^\bullet$, writing $S^n := S^\bullet([n])$.

\begin{defi}
    Fix a cosimplicial ring $S^\bullet$.
    \begin{enumerate}
        \item[(1)] For each $0 \leq i \leq n+1$, we let $p^n_i: S^n \to S^{n+1}$ denote the ring homomorphism induced by the unique injective order-preserving map $[n] \to [n+1]$ whose image omits the element $i$. We refer to the maps $p^n_i$ as the \emph{coface maps}.
        \item[(2)] For each $0 \leq i \leq n$, we let $\sigma_i^n: S^{n+1} \to S^n$ denote the ring homomorphism induced by the unique surjective order-preserving map $[n+1] \to [n]$ satisfying $(\sigma_i^n)^{-1}(i) = \{i, i+1\}$. We refer to the maps $\sigma_i^n$ as the \emph{codegeneracy maps}.
    \end{enumerate}
\end{defi}

\begin{defi}[Stratifications]
    For a given cosimplicial ring $S^\bullet$, a \emph{stratification} along $S^\bullet$ consists of a finitely generated projective $S^0$-module $M$, equipped with an $S^1$-linear isomorphism
    $$\ep: S^1 \otimes_{p_1^0, S^0} M \xrightarrow{\sim} S^1 \otimes_{p_0^0, S^0} M$$ 
    satisfying the following two conditions:
    \begin{enumerate}
        \item[(1)] The following cocycle diagram commutes:
        $$\xymatrixcolsep{0.2in}
        \xymatrix@C-2.3pc{ 
            & S^2\otimes_{p^1_2,S^1}S^1\otimes_{p_1^0,S^0}M \ar[rr]^-{id_{S^2}\otimes_{p^1_2,S^1}\ep} \ar@{=}[d] && S^2\otimes_{p^1_2,S^1}S^1\otimes_{p^0_0,S^0}M \ar@{=}[d] & \\
            S^2\otimes_{p^1_1,S^1}S^1\otimes_{p^0_1,S^0}M \ar[ddr]_-{id_{S^2}\otimes_{p^1_1,S^1}\ep} \ar@{=}[r] & S^2\otimes_{q_0,S^0}M && S^2\otimes_{q_1,S^0}M \ar@{=}[r] & S^2\otimes_{p^1_0,S^1}S^1\otimes_{p^0_1,S^0}M \ar[ldd]^-{id_{S^2}\otimes_{p^1_0,S^1}\ep} \\
            && S^2\otimes_{q_2,S^0}M \ar@{=}[dl]\ar@{=}[dr] && \\
            & S^2\otimes_{p^1_1,S^1}S^1\otimes_{p^0_0,S^0}M && S^2\otimes_{p^1_0,S^1}S^1\otimes_{p^0_0,S^0}M &
        }$$
        where $q_i: S^0 \to S^2$ is the homomorphism induced by the unique order-preserving map $[0] \to [2]$ sending $0$ to $i$.
        \item[(2)] The map obtained by applying the degeneracy $\sigma_0^0$ to $\ep$, namely $\sigma_0^0 \otimes \ep: M \to M$, is the identity $id_M$.
    \end{enumerate}
\end{defi}

\subsection[The functor M-Prism]{The functor $\M_{\Prism}$}

Let $(\C,\inte_{\C})$ be a ringed site.

\begin{defi}
    A \emph{$\C$-crystal of rank $r$} is a sheaf of $\inte_{\C}$-modules $\mathbb{H}$ satisfying the following conditions:
    \begin{enumerate}
        \item[(1)] For each object $X\in \C$, the evaluation $\mathbb{H}(X)$ is a locally free $\inte_{\C}(X)$-module of constant rank $r$.
        \item[(2)] For each morphism $f: X \to Y$ in $\C$, the canonical restriction map induces an isomorphism:
        $$ \inte_{\C}(X) \otimes_{\inte_{\C}(Y)} \mathbb{H}(Y) \xrightarrow{\sim} \mathbb{H}(X). $$
    \end{enumerate}
\end{defi}

Assume that $\C$ admits arbitrary finite products. Fix an object $X \in \C$ and let $S^n := \inte_{\C}(X^{n+1})$, where $X^{n+1}$ denotes the product of $n+1$ copies of $X$. It is a standard fact that $S^\bullet$ is naturally equipped with a cosimplicial ring structure; this is called the \emph{\v{C}ech nerve} of $X$.

\begin{prop}
    Given the data above, there is a natural functor from the category of $\C$-crystals to the category of stratifications along the \v{C}ech nerve of $X$.
\end{prop}

\begin{proof}
    This functor is defined by evaluating a given crystal $\mathbb{H}$ on $X$ to obtain the module $M := \mathbb{H}(X)$. The stratification isomorphism $\ep$ is naturally induced by the transposition morphism $X \times X \to X \times X$ sending $(x,y)$ to $(y,x)$.
\end{proof}

We now recall several well-known results regarding the equivalence between crystals and stratifications.

\begin{defi}
    An object $U \in \C$ is called a \emph{cover of the final object} if for any object $T \in \C$, there exists a covering $\{T' \to T\}$ such that $\mathrm{Hom}_{\C}(T',U) \neq \emptyset$.
\end{defi}

\begin{theo}[{\cite[Proposition 4.8]{Tian_2023}}]\label{thm: crystal to strat}
    If $T$ is a cover of the final object, then the natural functor from the category of $\C$-crystals to the category of stratifications along the \v{C}ech nerve of $T$ is an equivalence of categories.
\end{theo}

\begin{theo}\label{coverprism}
    If the ideal $\a/p$ is generated by a Koszul-regular sequence in $R_0$, then $(P,\delta)$ is a cover of the final object in the prismatic site $(B/A)_{\Prism}$.
\end{theo}

\begin{proof}
    By \cite[Theorem 4.3.6]{bhatt2022absolute}, the quotient $P/p$ is isomorphic to the derived prismatic cohomology $\Prism_{B/R} \otimes^{\mathbf{L}}_R R/p$. Consider the conjugate filtration on $\Prism_{B/R} \otimes^{\mathbf{L}}_R R/p$. By the Hodge--Tate comparison theorem, each graded piece of the conjugate filtration is flat over $B$. Consequently, $P/p \cong \Prism_{B/R} \otimes^{\mathbf{L}}_R R/p$ is flat over $B$.

    Given this flatness, the remainder of the proof follows identically to the arguments found in \cite[Lemma 2.21 and Proposition 2.22]{ogus2023crystalline}, \cite[Theorem 3.13]{Bhatt_2022}, or \cite{Tian_2023}.
\end{proof}

As a consequence of the preceding theorems, we can now formally define the functor $\M_{\Prism}$. Let $P^\bullet$ denote the \v{C}ech nerve of $P$ in the prismatic site $(B/A)_{\Prism}$. Note that $P^0 = P$, and $P^1$ is precisely the coproduct $P_{1,2}$ whose structure we computed explicitly in Theorem \ref{compare D}.

\begin{defi}
    We define $\M_{\Prism}$ to be the natural evaluation functor:
    $$ \M_{\Prism}: \mathbf{Crys}((B/A)_{\Prism}) \xrightarrow{\sim} \mathbf{Strat}(P^\bullet) $$
    from the category of prismatic crystals on $(B/A)_{\Prism}$ to the category of stratifications along the cosimplicial ring $P^\bullet$. Since $P$ is a cover of the final object, $\M_{\Prism}$ is an equivalence of categories.
\end{defi}

\subsection[Stratifications and p-connections]{Stratifications and $p$-connections}

The equivalence between stratifications and specific classes of connections is a standard result. Our notation largely follows that of Wang \cite{wang2024prismatic}.

In what follows, we let $X_{j,i}$ denote the element $1 \otimes \dots \otimes 1 \otimes X_j \otimes 1 \otimes \dots \otimes 1$ in the $n$-fold completed tensor product $R \widehat{\otimes}_A \dots \widehat{\otimes}_A R$, where $X_j$ is placed in the $(i+1)$-th tensor factor. Let $\underline{X}_i := (X_{1,i}, \dots, X_{d,i})$ be the corresponding tuple.

Consider the \v{C}ech nerve of the object $P$. It is naturally isomorphic to the following cosimplicial ring:
$$P^\bullet = P[\underline{Y}_1, \underline{Y}_2, \dots, \underline{Y}_\bullet]^\wedge_{\mathrm{pd}},$$ 
where 
$$\underline{Y}_i := \frac{\underline{X}_i - \underline{X}_0}{p}.$$ 
The face maps $p_i^n$ and degeneracy maps $\sigma_i^n$ are explicitly described as follows. For any $f \in P$,
$$ p_{i}^n(f) = 
\begin{cases} 
    f & \text{if } i > 0, \\ 
    \sum_{\alpha\in\N^d} p^{|\alpha|} \partial_{\alpha}(f) \underline{Y}_1^{[\alpha]} & \text{if } i = 0. 
\end{cases} $$
For the variables $\underline{Y}_j$, the face maps are:
$$ p_{i}^n(\underline{Y}_j) = 
\begin{cases} 
    \underline{Y}_{j+1} - \underline{Y}_1 & \text{if } i = 0, \\ 
    \underline{Y}_{j+1} & \text{if } 0 < i \leq j, \\ 
    \underline{Y}_j & \text{if } i > j. 
\end{cases} $$
The degeneracy maps are:
$$ \sigma_{i}^n(\underline{Y}_j) = 
\begin{cases} 
    0 & \text{if } (i,j) = (0,1), \\ 
    \underline{Y}_{j-1} & \text{if } i < j \text{ and } (i,j) \neq (0,1), \\ 
    \underline{Y}_j & \text{if } i \geq j. 
\end{cases} $$

Suppose we are given a stratification $(M,\ep)$. We can view both $P^1 \otimes_{p_1^0,P} M$ and $P^1 \otimes_{p^0_0,P} M$ as the completed PD-polynomial module $M[\underline{Y}_1]^\wedge_{\mathrm{pd}}$. Thus, $\ep$ can be viewed as an $A$-linear automorphism of $M[\underline{Y}_1]^\wedge_{\mathrm{pd}}$. By $P^1$-linearity, it is uniquely determined by its restriction to $M$. Suppose $\ep$ is of the form
$$\ep(m) = \sum_{\alpha \in \N^d} \theta_\alpha(m) \underline{Y}_1^{[\alpha]} \quad \text{for all } m \in M.$$
We note that $\theta_{\alpha}$ tends to $0$ $p$-adically as $|\alpha| \to \infty$.

\begin{rmk}
    There is a unique algebra automorphism of $P[\underline{Y}_1]^\wedge_{\mathrm{pd}}$ sending $\underline{Y}_1$ to itself and $f \in P$ to 
    $$\sum_{\alpha\in\N^d} p^{|\alpha|} \partial_{\alpha}(f) \underline{Y}_1^{[\alpha]}.$$
    This automorphism is compatible with the transposition $\iota: (x \otimes y) \mapsto (y \otimes x)$ on $R \widehat{\otimes}_A R$. Since $\iota \circ p_0^0 = p_1^0$, this identifies $P^1 \otimes_{p^0_0,P} M$ with $M[\underline{Y}_1]^\wedge_{\mathrm{pd}}$, but with the scalar multiplication twisted as $a \cdot m := \iota(a)m$.
\end{rmk}

Using the same method as in the above remark, we can also identify $P^2 \otimes_{p^1_i,P^1} M$ with $M[\underline{Y}_1, \underline{Y}_2]^\wedge_{\mathrm{pd}}$ for $i=0,1,2$.

\begin{prop}\label{equiv theta}
    Suppose $\{\theta_{\alpha}\}_{\alpha \in \N^d}$ is a family of $A$-linear endomorphisms of $M$ such that $\theta_{\alpha} \to 0$ as $|\alpha| \to \infty$. Then the map
    $$\ep(m) = \sum_{\alpha \in \N^d} \theta_\alpha(m) \underline{Y}_1^{[\alpha]}$$ 
    defines a stratification if and only if:
    \begin{enumerate}
        \item[(1)] $\theta_{\underline{0}} = \mathrm{id}_M$, which is equivalent to $\sigma_{0}^0 \otimes \ep = \mathrm{id}_M$;
        \item[(2)] For $1 \leq i \leq d$, let $e_i \in \N^d$ be the $i$-th standard basis vector. Then for all $m \in M$ and $f \in P$:
        \begin{align*}
            \begin{cases}
                \theta_{e_i}(fm) = f\theta_{e_i}(m) + p\partial_i(f)m, \\ 
                \theta_{e_i} \circ \theta_{e_j} = \theta_{e_j} \circ \theta_{e_i} \quad \forall i,j.
            \end{cases}
        \end{align*}
        \item[(3)] For any $\gamma = (\gamma_1, \gamma_2, \dots, \gamma_d) \in \N^d$, $\theta_{\gamma} = \prod_{i=1}^d \theta_{e_i}^{\gamma_i}$.
    \end{enumerate}
\end{prop}

\begin{proof}
    By definition, the map $\sigma^0_0: P^1 \to P$ acts by evaluating at $\underline{Y}_1 = 0$. Thus $\sigma^0_0 \otimes \ep$ is exactly $\theta_{\underline{0}}$, and the equivalence (1) follows immediately.

    Under the identifications established above, we compute the three pullbacks of $\ep$ to $P^2$:
    \begin{align*}
        id_{P^2} \otimes_{p^1_2, P^1} \ep(\underline{Y}_1^{[u]} \underline{Y}_2^{[v]} m) &= \sum_{\alpha \in \N^d} \theta_\alpha(m) \underline{Y}_1^{[\alpha]} \underline{Y}_1^{[u]} \underline{Y}_2^{[v]}, \\
        id_{P^2} \otimes_{p^1_0, P^1} \ep(\underline{Y}_1^{[u]} \underline{Y}_2^{[v]} m) &= \sum_{\alpha \in \N^d} \theta_\alpha(m) (\underline{Y}_2 - \underline{Y}_1)^{[\alpha]} \underline{Y}_1^{[u]} \underline{Y}_2^{[v]}, \\
        id_{P^2} \otimes_{p^1_1, P^1} \ep(\underline{Y}_1^{[u]} \underline{Y}_2^{[v]} m) &= \sum_{\alpha \in \N^d} \theta_\alpha(m) \underline{Y}_2^{[\alpha]} \underline{Y}_1^{[u]} \underline{Y}_2^{[v]}.
    \end{align*}
    Therefore, the cocycle descent condition is equivalent to:
    \begin{align*}
        &\sum_{\alpha \in \N^d} \sum_{\beta \in \N^d} \theta_\beta(\theta_\alpha(m)) (\underline{Y}_2 - \underline{Y}_1)^{[\beta]} \underline{Y}_1^{[\alpha]} = \sum_{\alpha \in \N^d} \theta_\alpha(m) \underline{Y}_2^{[\alpha]} \\
        \Longleftrightarrow \quad & \sum_{\alpha \in \N^d} \theta_\alpha(m) \underline{Y}_2^{[\alpha]} = \sum_{\alpha, \beta \in \N^d} \theta_\beta(\theta_\alpha(m)) \sum_{k+l=\beta} (-1)^{|l|} \underline{Y}_2^{[k]} \underline{Y}_1^{[l]} \underline{Y}_1^{[\alpha]} \\
        \Longleftrightarrow \quad & \theta_\gamma(m) = \sum_{\alpha, \beta \in \N^d} (-1)^{|\beta|} \theta_{\beta+\gamma}(\theta_\alpha(m)) \underline{Y}_1^{[\alpha]} \underline{Y}_1^{[\beta]} \quad \text{for all } \gamma \in \N^d. \label{con:000} \tag{*}
    \end{align*}

    On the other hand, since $\ep$ is $P^1$-linear, we have:
    \begin{align*}
        \ep(fm) &= \sum_{\alpha \in \N^d} \theta_\alpha(fm) \underline{Y}_1^{[\alpha]} \\
        &= \left( \sum_{\alpha \in \N^d} p^{|\alpha|} \partial_{\alpha}(f) \underline{Y}_1^{[\alpha]} \right) \left( \sum_{\alpha \in \N^d} \theta_\alpha(m) \underline{Y}_1^{[\alpha]} \right) \\
        &= \sum_{\alpha \in \N^d} \underline{Y}_1^{[\alpha]} \sum_{\beta \leq \alpha} p^{|\beta|} \binom{\alpha}{\beta} \partial_{\beta}(f) \theta_{\alpha-\beta}(m).
    \end{align*}
    Comparing coefficients, $(M,\ep)$ is a stratification if and only if $\theta_{\underline{0}} = \mathrm{id}_M$, the descent condition \eqref{con:000} holds, and for any $\gamma \in \N^d$, $f\in P$, and $m \in M$, we have:
    \begin{equation*}\tag{**}\label{con:001}
        \theta_{\gamma}(fm) = \sum_{\beta \leq \gamma} p^{|\beta|} \binom{\gamma}{\beta} \partial_{\beta}(f) \theta_{\gamma-\beta}(m).
    \end{equation*}

    Taking $\gamma = e_i$, equation \eqref{con:001} simplifies to $\theta_{e_i}(fm) = f\theta_{e_i}(m) + p\partial_i(f)m$. By comparing the coefficients of $\underline{Y}_1^{[e_i]}$ in \eqref{con:000}, we obtain $\theta_{\gamma+e_i} = \theta_{\gamma} \circ \theta_{e_i}$. Setting $\gamma = e_j$, this yields the commutativity $\theta_{e_i} \circ \theta_{e_j} = \theta_{e_j} \circ \theta_{e_i}$. By induction, we deduce $\theta_{\gamma} = \prod_{i=1}^d \theta_{e_i}^{\gamma_i}$.

    It remains to verify the converse: that the commutativity $\theta_{e_i} \circ \theta_{e_j} = \theta_{e_j} \circ \theta_{e_i}$ and the product formula $\theta_{\gamma} = \prod_{i=1}^d \theta_{e_i}^{\gamma_i}$ imply both \eqref{con:000} and \eqref{con:001}. Equation \eqref{con:000} follows from a straightforward combinatorial cancellation:
    \begin{align*}
        \sum_{\alpha,\beta \in \N^d} (-1)^{|\beta|} \theta_{\beta+\gamma}(\theta_{\alpha}(m)) \underline{Y}_1^{[\alpha]} \underline{Y}_1^{[\beta]} 
        &= \sum_{\alpha,\beta \in \N^d} (-1)^{|\beta|} \underline{\theta}^{\beta} \underline{\theta}^\gamma \underline{\theta}^\alpha \binom{\alpha+\beta}{\alpha} \underline{Y}_1^{[\alpha+\beta]} \\
        &= \theta_{\gamma} \circ \sum_{\alpha,\beta \in \N^d} (-1)^{|\beta|} \underline{\theta}^{\beta} \underline{\theta}^\alpha \binom{\alpha+\beta}{\alpha} \underline{Y}_1^{[\alpha+\beta]} \\
        &= \theta_\gamma(m).
    \end{align*}
    Equation \eqref{con:001} follows by induction on $|\gamma|$. With the convention that $\binom{a}{b}=0$ for $b<0$, we have:
    \begin{align*}
        \theta_{\gamma+e_i}(fm) &= \theta_{\gamma}\left(f\theta_{e_i}(m) + p\partial_i(f)m\right) \\
        &= \sum_{\beta \leq \gamma} p^{|\beta|} \binom{\gamma}{\beta} \partial_{\beta}(f) \theta_{\gamma-\beta}(\theta_{e_i}(m)) + \sum_{\beta \leq \gamma} p^{|\beta+e_i|} \binom{\gamma}{\beta} \partial_{\beta+e_i}(f) \theta_{\gamma-\beta}(m) \\
        &= \sum_{\beta \leq \gamma} p^{|\beta|} \left[ \binom{\gamma+e_i}{\beta} - \binom{\gamma}{\beta-e_i} \right] \partial_{\beta}(f) \theta_{\gamma+e_i-\beta}(m) \\
        &\quad + \sum_{e_i \leq \beta \leq \gamma+e_i} p^{|\beta|} \binom{\gamma}{\beta-e_i} \partial_{\beta}(f) \theta_{\gamma-\beta+e_i}(m) \\
        &= \sum_{\beta \leq \gamma+e_i} p^{|\beta|} \binom{\gamma+e_i}{\beta} \partial_{\beta}(f) \theta_{\gamma+e_i-\beta}(m).
    \end{align*}
    This completes the proof.
\end{proof}

\begin{rmk}
    Note that the formula for $\ep$ can be formally rewritten as an exponential map:
    \begin{align*}
        \sum_{\alpha \in \N^d} \theta_\alpha(m) \underline{Y}_1^{[\alpha]} = \left( \prod_{i=1}^d \exp(\theta_{e_i} Y_{1,i}) \right)(m),
    \end{align*}
    where $Y_{1,i}$ is the $i$-th component of $\underline{Y}_1$.
\end{rmk}

Hence, we arrive at the following theorem:

\begin{theo}\label{localDprism}
    \begin{enumerate}
        \item[(1)] There is an equivalence of categories:
        $$ \text{stratifications along } P^\bullet
        \cong \text{integrable topologically quasi-nilpotent $p$-connections on } P. $$
        \item[(2)] If $\a$ is locally generated, modulo $p$, by a Koszul-regular sequence, the composition of the above equivalence with the functor $\M_{\Prism}$ provides a rank-preserving equivalence of categories:
        $$ \D_\Prism: \mathbf{Crys}_{(B/A)_{\Prism}} \xrightarrow{\sim}
        \text{integrable topologically quasi-nilpotent $p$-connections on } P. $$
    \end{enumerate}
\end{theo}
\begin{proof}
    (1) Let $M$ be a finitely generated projective $P$-module. By Proposition \ref{equiv theta}, endowing $M$ with a stratification $\ep$ along $P^\bullet$ is equivalent to specifying a family of $A$-linear endomorphisms $\{\theta_\alpha\}_{\alpha \in \N^d}$ satisfying the given descent conditions. Specifically, the stratification is uniquely determined by the degree-one operators $\theta_{e_i}: M \to M$. The relation $\theta_{e_i}(fm) = f\theta_{e_i}(m) + p\partial_i(f)m$ precisely means that the operators $\theta_{e_i}$ define a $p$-connection on $M$ over $P$. The commutativity condition $\theta_{e_i} \circ \theta_{e_j} = \theta_{e_j} \circ \theta_{e_i}$ is equivalent to the integrability of this $p$-connection. Finally, the product formula $\theta_\gamma = \prod_{i=1}^d \theta_{e_i}^{\gamma_i}$ combined with the decay condition ($\theta_\alpha \to 0$ $p$-adically as $|\alpha| \to \infty$) is exactly the definition of topological quasi-nilpotence for the $p$-connection. This establishes the first equivalence.

    (2) By Theorem \ref{coverprism}, the Koszul-regularity assumption ensures that the prismatic envelope $P$ is a cover of the final object in the prismatic site $(B/A)_{\Prism}$. Consequently, Theorem \ref{thm: crystal to strat} guarantees that the natural evaluation functor $\M_{\Prism}$ induces an equivalence of categories between prismatic crystals on $(B/A)_{\Prism}$ and stratifications along the \v{C}ech nerve $P^\bullet$. Composing this equivalence with the one established in part (1) yields the functor $\mathbb{D}_{\Prism}$. Since both intermediate equivalences preserve the underlying locally free modules, the composed equivalence is manifestly rank-preserving.
\end{proof}
We will denote the category of integrable topologically quasi-nilpotent $p$-connections on $P$ by $p\text{-}\mathbf{MIC}_{P}^{\nil}$.

\subsection{Independence of choice of chart}

The method in this subsection follows Wang \cite{wang2024prismatic}. We want to show that the construction does not depend on the choice of
$$\square:A\langle\underline{X}^{\pm1}\rangle\to R.$$

Suppose $J=\ker(\sigma^0_0:P^1\to P)$ and $I=\ker(R\widehat{\otimes}_AR\to R)$. Then there is an isomorphism $I/I^2\cong \Omega_{R/A_0}$.

\begin{lemma}\label{lem:free-diff}
    There is a unique isomorphism $$P\otimes_R I/I^2\cong J/J^{[2]}$$ mapping $[a\otimes 1 - 1\otimes a] \in I/I^2$ to the unique element $x \in J/J^{[2]}$ such that $px = [a\otimes 1 - 1\otimes a]$. Here, $J^{[2]}$ denotes the second PD-power of $J$. Via the standard identification $I/I^2 \cong \widehat{\Omega}_{R/A}$, this yields $P\otimes_R\widehat{\Omega}_{R/A} \cong J/J^{[2]}$.
\end{lemma}

\begin{proof}
    For any $a\in R$, the class $[a\otimes1-1\otimes a]\in J/J^{[2]}$ lies in $pJ/J^{[2]}$. Since $J/J^{[2]}$ is a free $P$-module, there is a unique element $x\in J/J^{[2]}$ such that $px=[a\otimes1-1\otimes a]$. The assignment $[a\otimes 1-1\otimes a]\mapsto x$ is $R$-linear and hence induces a homomorphism $P\otimes_R I/I^2\to J/J^{[2]}$. The calculations on a toric chart above show that this is an isomorphism.
\end{proof}

For any stratification $(M,\ep)$, define $\theta':M\to M\otimes_P\Omega_P$ as follows. For $m\in M$, take the image of $(\ep(m)-1\otimes_{p^0_0}m)\in J\otimes_{p^0_0,P}M$ in $J/J^{[2]}\otimes_PM\cong M\otimes_P\Omega_P$.

\begin{theo}\label{charttheta}
    We have $\theta=\theta'$.
\end{theo}

\begin{proof}
    By Proposition \ref{equiv theta}, the stratification isomorphism is given by the expansion $$\ep(m) = \sum_{\alpha \in \N^d} \theta_\alpha(m) \underline{Y}_1^{[\alpha]}.$$ Modulo $J^{[2]}$, all higher-degree terms with $|\alpha| \ge 2$ vanish. Therefore, the projection of $\ep(m) - 1\otimes_{p^0_0} m$ into $J/J^{[2]} \otimes_P M$ is precisely the degree-one part $\sum_{i=1}^d \theta_{e_i}(m) \otimes \underline{Y}_1^{[e_i]}$. Under the isomorphism $J/J^{[2]} \cong \Omega_P$ from Lemma \ref{lem:free-diff}, the basis element $\underline{Y}_1^{[e_i]}$ corresponds exactly to $dX_i$. Thus, we obtain $\theta'(m) = \sum_{i=1}^d \theta_{e_i}(m) \otimes dX_i$, which is exactly the definition of the connection $1$-form $\theta(m)$ associated to the operators $\{\theta_{e_i}\}$.
\end{proof}

\subsection{Cohomological comparison: Local version}\label{localcohcompare}

We maintain the notation from the previous subsections. Let $\X=\Spec(R_0/\a)$ and $Y=\Spf(R)$. Recall that there is a Frobenius lift $\phi_R$ on $R$. Fix $\mathbb{H}\in\mathbf{Crys}_{(\X/A)_{\Prism}}$ and let $(\mathcal{M},\nabla_{\mathcal{M}})=\D_{\Prism}(\mathbb{H})$. Let $\DR(\mathcal{M},\nabla_{\mathcal{M}})$ be the corresponding de Rham complex. In the calculations below, it will be more convenient to use the twisted complex
$$\DR(\mathcal M,p^{-1}\nabla_{\mathcal M}):
\mathcal M\to \mathcal M\otimes_P\widehat\Omega^1_{R/A}\{-1\}\to
\mathcal M\otimes_P\widehat\Omega^2_{R/A}\{-2\}\to \cdots,$$
where $\widehat\Omega^i_{R/A}\{-i\}$ denotes the same underlying module as $\widehat\Omega^i_{R/A}$, but with the differential divided by the corresponding power of $p$. Equivalently, the differential is induced by $p^{-1}\nabla_{\mathcal M}$. Since tensoring the $i$-th term with the formal symbol $p^{-i}$ only renormalizes the grading and does not change the underlying de Rham data, the natural identification
$$\mathcal M\otimes_P\widehat\Omega^i_{R/A}\cong
\mathcal M\otimes_P\widehat\Omega^i_{R/A}\{-i\}$$
identifies $\DR(\mathcal M,\nabla_{\mathcal M})$ with $\DR(\mathcal M,p^{-1}\nabla_{\mathcal M})$ up to this harmless Tate twist. Thus proving the comparison for the twisted complex is equivalent to proving it for $\DR(\mathcal M,\nabla_{\mathcal M})$.
We will prove the following theorem:

\begin{theo}\label{localcoh}
    There is an explicitly constructed quasi-isomorphism
    $$\rho_{R,\phi_R}:R\Gamma\left((\X/A)_{\Prism},\mathbb{H}\right)\simeq \DR(\mathcal{M},\nabla_{\mathcal{M}})$$
\end{theo}

\begin{rmk}
    This theorem has already been proved by Ogus in \cite[Theorem 6.3.5]{ogus2023crystalline}. We give a proof following Tian's method (see \cite{Tian_2023} or \cite[Subsection 4.3]{wang2024prismatic}). This method is more elementary and explicit, and it allows us to understand the behavior of $\rho_{R,\phi_R}$ when we change $\phi_R$.
\end{rmk}

We construct an auxiliary complex and show that it is explicitly quasi-isomorphic to both $R\Gamma((\X/A)_{\Prism},\mathbb H)$ and $\DR(\mathcal M,\nabla_{\mathcal M})$.

Let
$$\widehat{\Omega}_{P^m/P}^1=\bigoplus_{1\leq i\leq m,\ 1\leq j\leq d}P^m\dif Y_{i,j}$$
be the complete module of PD-differentials. Hence there is a differential operator
$$\dif:P^m\to\widehat{\Omega}_{P^m/P}^1$$
which is $P$-linear.
Let
$$\widehat{\Omega}_{P^m}^1=(P^m\otimes_{R}\widehat{\Omega}_{R/A}^1\{-1\})\oplus \widehat{\Omega}_{P^m/P}^1.$$
It is free with basis
$$p^{-1}\dif X_1,p^{-1}\dif X_2,\dots,p^{-1}\dif X_d, \dif Y_{1,1},\dif Y_{1,2},\dots,\dif Y_{m,d}.$$
Moreover, for any order-preserving map $f:[m_1]\to [m_2]$, the cosimplicial structure induces a map $P^{m_1}\to P^{m_2}$, and functoriality of differential modules gives a canonical map
$$f_*:\widehat{\Omega}^1_{P^{m_1}}\to \widehat{\Omega}^1_{P^{m_2}}.$$
This makes $\widehat{\Omega}^1_{P^{\bullet}}$ a cosimplicial $P^\bullet$-module. For any $j>0$, let $\widehat{\Omega}^j_{P^{\bullet}}$ be its $j$-fold wedge product.
We may write
$$\widehat{\Omega}^j_{P^{m}}=\bigoplus_{k=0}^j\widehat{\Omega}^k_{P}\{-k\}\otimes_{P}\widehat{\Omega}^{j-k}_{P^{m}/P}.$$
For any $\omega_k\in\widehat{\Omega}^k_{P}\{-k\}$ and $\eta_{j-k}\in \widehat{\Omega}^{j-k}_{P^{m}/P}$, put
$$\dif_P(\omega_k\otimes \eta_{j-k})=\dif(\omega_k)\otimes \eta_{j-k}+(-1)^{k}\omega_{k}\otimes \dif(\eta_{j-k}),$$
where $\dif:\widehat{\Omega}^k_{P}\{-k\}\to \widehat{\Omega}^{k+1}_{P}\{-k-1\}$ is the twisted differential introduced in Corollary \ref{higherpdif}.

Now suppose $\mathbb{H}\in\mathbf{Crys}_{(B/A)_{\Prism}}$. Let $q_0:P=P^0\to P^n$ be the map induced by the inclusion $\{0\}\subset\{0,1,2,\dots,n\}$. Let $(\mathcal{M},\nabla)$ be the corresponding $p$-connection and $(\mathcal{M},\ep)$ be the induced stratification. We have a canonical isomorphism
$$\mathcal{M}\otimes_{P,q_0}P^\bullet\cong\mathbb{H}(P^\bullet).$$

For any $r,s$, we define
$$\dif_1^{r,s}:=\sum_{i=0}^{r+1}(-1)^ip_i^{r}:\mathbb{H}(P^r)\otimes_{P^r}\widehat{\Omega}_{P^r}^s\to \mathbb{H}(P^{r+1})\otimes_{P^{r+1}}\widehat{\Omega}_{P^{r+1}}^s$$
where $p_i^r$ acts as the tensor product of the coface maps $p_i^r$ on both factors. This is the horizontal differential in the totalization of the cosimplicial module $\mathbb{H}(P^\bullet)\otimes_{P^\bullet}\widehat{\Omega}_{P^\bullet}^s$. We also define
$$\dif_2^{r,s}(x\otimes\omega)=\nabla(x)\wedge \omega+x\otimes \dif_P(\omega).$$

\begin{theo}\label{localbicomplex}
    For any $s\geq 0$, $\dif_2^{\bullet,s}$ defines a morphism of cosimplicial modules. Hence $(\mathbb{H}(P^r)\otimes_{P^r}\widehat{\Omega}_{P^r}^s,\dif_1^{r,s},\dif_2^{r,s})$ defines a bicomplex.
\end{theo}

\begin{proof}
    We follow the explicit calculation of \cite[Lemma 4.14]{wang2024prismatic}; the case of the structure crystal is treated in \cite[Lemma 4.13]{wang2024prismatic}. 
    Write
    $$e_i:=p^{-1}\dif X_i,\qquad \eta_{j,i}:=\dif Y_{j,i}.$$
    Then $\widehat\Omega^1_{P^r}$ is freely generated over $P^r$ by the $e_i$ and the $\eta_{j,i}$ for $1\leq j\leq r$. If $p_a^r:P^r\to P^{r+1}$ is a coface map with $a>0$, then it preserves the zeroth vertex. Hence it does not change the chosen trivialization of the crystal, and for $m\in \mathcal M$ one has
    $$p_a^r(m)=m,\qquad p_a^r(e_i)=e_i.$$
    The variables $Y_{j,i}$ are sent to the corresponding variables described by the \v{C}ech formulas above. Therefore the equality
    $$p_a^r\circ \dif_2^{r,s}=\dif_2^{r+1,s}\circ p_a^r$$
    follows from the ordinary functoriality of PD-differentials and from the Leibniz rule.

    The only nontrivial coface is $p_0^r$; more generally, the same calculation treats any order-preserving map $f:[r]\to[r']$ with $f(0)\ne0$. Under the identification
    $$\mathbb H(P^r)\cong P^r\otimes_{q_0,P}\mathcal M,$$
    the map $p_0^r$ is obtained by changing the base vertex from $0$ to $1$, and therefore acts through the stratification. If
    $$\epsilon(m)=\sum_{\alpha\in\mathbb N^d}\theta_\alpha(m)\underline Y_1^{[\alpha]},$$
    then
    $$p_0^r(m)=\sum_{\alpha\in\mathbb N^d}\theta_\alpha(m)\underline Y_1^{[\alpha]}.$$
    Moreover,
    $$p_0^r(e_i)=e_i+\eta_{1,i},\qquad
    p_0^r(Y_{j,i})=Y_{j+1,i}-Y_{1,i}.$$
    Equivalently, for a general order-preserving map $f$ with $f(0)\ne0$, the stratification gives
    $$f(m)=\sum_{\alpha\in\mathbb N^d}\theta_\alpha(m)\underline Y_{f(0)}^{[\alpha]},
    \qquad
    f(e_i)=e_i+\dif Y_{f(0),i}.$$
    It suffices to check the equality on $m\in\mathcal M$, because both sides satisfy the same graded Leibniz rule with respect to $P^r$ and the exterior algebra. Since the $p$-connection is given by
    $$\nabla(m)=\sum_{i=1}^d\theta_{e_i}(m)e_i,$$
    Proposition \ref{equiv theta} gives $\theta_\alpha=\theta_{e_1}^{\alpha_1}\cdots\theta_{e_d}^{\alpha_d}$ and the endomorphisms $\theta_{e_i}$ commute. Thus
    \begin{align*}
    \dif_2^{r+1,0}\bigl(p_0^r(m)\bigr)
    &=\dif_2^{r+1,0}\left(\sum_{\alpha}\theta_\alpha(m)\underline Y_1^{[\alpha]}\right)\\
    &=\sum_{\alpha,i}\theta_{e_i}\theta_\alpha(m)\underline Y_1^{[\alpha]}e_i
      +\sum_{\alpha,i}\theta_\alpha(m)\underline Y_1^{[\alpha-e_i]}\eta_{1,i}\\
    &=\sum_{\alpha,i}\theta_{\alpha+e_i}(m)\underline Y_1^{[\alpha]}e_i
      +\sum_{\alpha,i}\theta_{\alpha+e_i}(m)\underline Y_1^{[\alpha]}\eta_{1,i}\\
    &=\sum_i\left(\sum_{\alpha}\theta_\alpha(\theta_{e_i}(m))\underline Y_1^{[\alpha]}\right)(e_i+\eta_{1,i})\\
    &=p_0^r\left(\sum_i\theta_{e_i}(m)e_i\right)
    =p_0^r\bigl(\dif_2^{r,0}(m)\bigr).
    \end{align*}
    Here $\underline Y_1^{[\alpha-e_i]}$ is interpreted as $0$ if $\alpha_i=0$, and in the third line we reindex $\alpha$ by $\alpha+e_i$. This is exactly the cancellation in the Taylor expansion: the terms coming from differentiating the divided-power variables $\underline Y_1^{[\alpha]}$ are absorbed by the extra summand $\eta_{1,i}$ in $p_0^r(e_i)$.

    The same computation after wedging with an arbitrary form $\omega$ proves
    $$p_0^r\circ \dif_2^{r,s}=\dif_2^{r+1,s}\circ p_0^r$$
    for all $s$. Therefore $\dif_2^{\bullet,s}$ is a morphism of cosimplicial modules. Finally, $\dif_1^2=0$ is the usual cosimplicial identity, $\dif_2^2=0$ follows from the integrability of the $p$-connection and from $\dif_P^2=0$, and the equality just proved gives $\dif_1\dif_2+\dif_2\dif_1=0$ with the standard total-complex sign convention.
\end{proof}

We define the complex $\mathrm{Tot_{\mathbb{H}}}$ to be the totalization of $(\mathbb{H}(P^r)\otimes_{P^r}\widehat{\Omega}_{P^r}^s,\dif_1^{r,s},\dif_2^{r,s})$. We will prove that both $\mathbb{H}(P^\bullet)$ and $\DR(\mathcal{M},p^{-1}\nabla)$ are isomorphic to $\mathrm{Tot_{\mathbb{H}}}$. In fact, by our definition, we have
$$(\mathbb{H}(P^\bullet),\dif_1^{\bullet,0})$$
is the totalization of the cosimplicial module \(\mathbb{H}(P^\bullet)\) and
$$(\mathbb{H}(P^0)\otimes_{P^0}\widehat{\Omega}_{R^0}^\bullet,\dif_2^{0,\bullet})$$
is equal to $\DR(\mathcal{M},p^{-1}\nabla)$. Thus, the inclusions provide two morphisms of complexes
$$(\mathbb{H}(P^\bullet),\dif_1^{\bullet,0})\to \mathrm{Tot}_{\mathbb{H}}$$
and
$$\DR(\mathcal{M},p^{-1}\nabla)\to \mathrm{Tot}_{\mathbb{H}}.$$

\begin{theo}\label{localtotqi}
    The above two morphisms are indeed quasi-isomorphisms.
\end{theo}

\begin{proof}
    We give the \v{C}ech calculation explicitly. Put
    $$K^{r,s}:=\mathbb H(P^r)\otimes_{P^r}\widehat\Omega^s_{P^r}.$$
    The first assertion is equivalent to saying that the quotient bicomplex
    $$K^{\bullet,>0}:=\bigoplus_{s>0}K^{\bullet,s}[-s]$$
    is acyclic for the horizontal \v{C}ech differential. It is enough to check this after choosing the trivialization
    $$\mathbb H(P^r)\cong P^r\otimes_P\mathcal M,$$
    because $\mathcal M$ is finite projective and the transition maps are given by the stratification.

    We first treat $s=1$. The \v{C}ech complex of the degree-one differential modules is generated by the elements
    $$e_i=p^{-1}\dif X_i,\qquad \eta_{j,i}=\dif Y_{j,i}.$$
    The key identity is already visible in degree one:
    $$p_0^0(e_i)-p_1^0(e_i)=(e_i+\eta_{1,i})-e_i=\eta_{1,i}.$$
    Hence the relative differential $\eta_{1,i}$ in degree $1$ is a \v{C}ech boundary. More generally,
    $$\eta_{j,i}=q_j(e_i)-q_0(e_i),$$
    where $q_j:P\to P^r$ denotes the map corresponding to the $j$-th vertex. This is the same computation as in the standard infinitesimal \v{C}ech complex of differentials (cf. \cite[Example 2.16]{bhatt2011crystalline}). It gives an explicit contracting homotopy on the normalized \v{C}ech complex of $\widehat\Omega^1_{P^\bullet}$: on a normalized generator represented by $\eta_{j,i}$, the homotopy sends it to an alternating sum of the vertex differentials $q_\ell(e_i)$ whose \v{C}ech boundary is $\eta_{j,i}$. Thus the row
    $$K^{\bullet,1}:\quad
    K^{0,1}\to K^{1,1}\to K^{2,1}\to \cdots$$
    is null-homotopic.

    More explicitly, \cite[Example 2.16]{bhatt2011crystalline} applies to the cosimplicial module
    $$N_i^r:=\bigoplus_{0\leq j\leq r}P^r e_{j,i},
    \qquad f(e_{j,i})=e_{f(j),i},$$
    and $\widehat\Omega^1_{P^r}$ is the quotient of the direct sum of these $N_i^r$ by the constant summands generated by the diagonal differentials. Since Example 2.16 contracts $N_i^\bullet$, the quotient by the constant part contracts the relative generators $\eta_{j,i}=e_{j,i}-e_{0,i}$. This is the precise degree-one null-homotopy used above.

    For higher differential degree, we use that $\widehat\Omega^s_{P^r}$ is the derived exterior power $\wedge^s\widehat\Omega^1_{P^r}$ and that the modules involved are finite projective over the corresponding $P^r$. The contracting homotopy in degree one is compatible with the cosimplicial algebra structure and therefore extends, by the usual derivation rule, to the exterior algebra. This is the same passage from degree one to all degrees as in \cite[Lemma 4.15]{Tian_2023}. Consequently the \v{C}ech complex $K^{\bullet,s}$ is null-homotopic for every $s>0$. This proves that
    $$(\mathbb H(P^\bullet),\dif_1^{\bullet,0})\longrightarrow \mathrm{Tot}_{\mathbb H}$$
    is a quasi-isomorphism.

    It remains to compare the total complex with the zeroth column. We use the filtration by the horizontal \v{C}ech degree. The associated graded pieces are the de Rham complexes of the formal PD-polynomial algebras
    $$P^r\simeq P[\underline Y_1,\dots,\underline Y_r]_{\mathrm{pd}}^\wedge$$
    with coefficients in the pullback of $(\mathcal M,p^{-1}\nabla)$. The relative part in the variables $\underline Y_j$ is the standard PD-de Rham complex
    $$P[\underline Y_1,\dots,\underline Y_r]_{\mathrm{pd}}^\wedge
    \to
    \bigoplus_{j,i}P[\underline Y_1,\dots,\underline Y_r]_{\mathrm{pd}}^\wedge\dif Y_{j,i}\to\cdots,$$
    which is contractible by the usual PD Poincar\'e homotopy. In one variable the relative PD-de Rham complex is
    $$P\langle Y\rangle_{\mathrm{pd}}^\wedge
    \xrightarrow{\dif}
    P\langle Y\rangle_{\mathrm{pd}}^\wedge \dif Y,$$
    and the contracting homotopy of the positive-degree part is given explicitly by
    $$h\left(Y^{[n]}\dif Y\right)=Y^{[n+1]},\qquad h\left(Y^{[n]}\right)=0.$$
    Indeed, $\dif(Y^{[n+1]})=Y^{[n]}\dif Y$. For several variables the relative PD-de Rham complex is the completed tensor product of these one-variable complexes; equivalently, one applies the above homotopy successively to the first differential variable appearing in a monomial
    $$\underline Y^{[\alpha]}\dif Y_{j_1,i_1}\wedge\cdots\wedge \dif Y_{j_t,i_t}.$$
    This gives a homotopy between the identity and the projection to the constant term. This is the standard calculation used in \cite[Lemmas 4.13--4.15]{Tian_2023}. Hence the projection $P^r\to P$ induces a quasi-isomorphism from the column in \v{C}ech degree $r$ to the de Rham complex over $P$.

    The compatibility of these homotopies with the coface maps identifies the resulting horizontal complex with the constant \v{C}ech complex attached to $\DR(\mathcal M,p^{-1}\nabla)$; this constant \v{C}ech complex is contracted by the extra degeneracy inserting the zeroth vertex. Therefore the inclusion
    $$\DR(\mathcal M,p^{-1}\nabla)\longrightarrow \mathrm{Tot}_{\mathbb H}$$
    is also a quasi-isomorphism. Using the identification of $\DR(\mathcal M,p^{-1}\nabla)$ with $\DR(\mathcal M,\nabla)$ explained above gives the desired statement.
\end{proof}

Finally, it remains to compare $\mathbb{H}(P^\bullet)$ and $R\Gamma\left((\X/A)_{\Prism},\mathbb{H}\right)$. This follows from \cite[Proposition 3.10]{Tian_2023}.

\subsection{Lifting of the Ogus--Vologodsky correspondence}\label{OV}

We regard $pA$ as a PD-ideal with its canonical divided-power structure. The local lifted Ogus--Vologodsky functor is first defined by the usual formula on connections. We then identify this definition with a purely stratification-theoretic construction; the latter is the form used in the descent argument proving the equivalence.

\begin{theo}\label{equiv}
    Let $X$ be a scheme over $A/p$ which is locally isomorphic to the quotient of a smooth scheme by an ideal generated by a Koszul-regular sequence. Then there is a canonical equivalence of categories
    $$\mathbf{Crys}_{(X^{(1)}/A)_{\Prism}}\cong \mathbf{Crys}_{(X/A)_{\mathrm{crys}}}.$$
    Here $X^{(1)}$ denotes the Frobenius twist of $X$ over $A/p$, and $(X/A)_{\mathrm{crys}}$ denotes the crystalline site of $X$ over $A$.
\end{theo}

Before constructing the functor, we record that the same local complete intersection hypothesis is preserved by Frobenius twist.

\begin{lemma}
    Let $A$ be a ring of characteristic $p$, let $R$ be a smooth $A$-algebra, and let $f_1,\ldots,f_r\in R$ be a Koszul-regular sequence. Put
    $$R^{(1)}:=A\otimes_{\phi_A,A}R,\qquad f_i^{(1)}:=1\otimes f_i.$$
    Then $R^{(1)}$ is smooth over $A$, and $f_1^{(1)},\ldots,f_r^{(1)}$ is a Koszul-regular sequence in $R^{(1)}$.
\end{lemma}

\begin{proof}
    Smoothness is stable under base change. The relative Frobenius
    $$\phi_{R/A}:R^{(1)}\to R,\qquad a\otimes x\mapsto ax^p,$$
    is faithfully flat because $R$ is smooth over $A$. Since
    $$\phi_{R/A}(f_i^{(1)})=f_i^p,$$
    and powers of a Koszul-regular sequence are Koszul-regular, the Koszul complex of $f_1^{(1)},\ldots,f_r^{(1)}$ becomes acyclic in positive degrees after the faithfully flat base change $\phi_{R/A}$. Hence, it is already acyclic in positive degrees over $R^{(1)}$.
\end{proof}

First we assume that $X=\operatorname{Spec}(B)$ is affine. We use the following category of presentations. An object of $\mathcal P_B$ is a surjection
$$T=(S\twoheadrightarrow B),$$
where $S$ is a $p$-complete smooth $\delta$-$A$-algebra. A morphism in $\mathcal P_B$ is a $\delta$-$A$-homomorphism over $B$. For such a presentation, set
$$S^{(1)}:=A\widehat\otimes_{\phi_A,A}S,\qquad B^{(1)}:=A/p\otimes_{\phi_{A/p},A/p}B,$$
and write
$$P_T:=\Prism_{B^{(1)}/S^{(1)}},\qquad D_T:=D_S(B),$$
where $P_T$ is the prismatic envelope of $B^{(1)}$ in $S^{(1)}$, and $D_T$ is the $p$-complete PD-envelope of $B$ in $S$.
By the universal property of prismatic envelopes, we have
\[P_T=A\widehat{\otimes}_{\phi_A,A}\Prism_{B/S}.\]
We denote by
$$\Phi_T:P_T\longrightarrow D_T$$
the composition of $$P_T=A\widehat{\otimes}_{\phi_A,A}\Prism_{B/S}\to S\widehat{\otimes}_{\phi_S,S}\Prism_{B/S}$$ and the canonical isomorphism $S\widehat{\otimes}_{\phi_S,S}\Prism_{B/S}\cong D_{S}(B)$ constructed in Corollary \ref{cor: crys comp}. This construction is functorial in $T$.

\begin{lemma}\label{lem:presentation-frobenius-map}
    For every $T=(S\twoheadrightarrow B)$ in $\mathcal P_B$, the canonical homomorphism
    $$\Phi_T:P_T\longrightarrow D_T$$
    is finite faithfully flat.
\end{lemma}

\begin{proof}
      The assertion is local on $S$. We may therefore choose elements $f_1,\ldots,f_r\in I_T:=\ker (S\to B)$ whose images form a Koszul-regular sequence in $S/p$ and generate $I_T/pS$. Let $A\{x_1,\ldots,x_r\}$ be the $\delta$-polynomial algebra over $A$, and define two $\delta$-$A$-homomorphisms
    $$o_1:A\{\underline{x}\}\to S^{(1)},\qquad x_i\mapsto 1\otimes f_i,$$
    $$o_2:A\{\underline{x}\}\to S,\qquad x_i\mapsto \phi_S(f_i).$$
    These are compatible with the relative Frobenius
    $$\phi_{S/A}:S^{(1)}\to S,\qquad 1\otimes s\mapsto \phi_S(s),$$
    in the sense that $o_2=\phi_{S/A}\circ o_1$. Let $A\{y_1,\ldots,y_r\}$ be the $\delta$-$A\{\underline{x}\}$-algebra defined by $x_i\mapsto py_i$. By Proposition \ref{prop:prism_close} and Corollary \ref{cor: crys comp}, we have canonical identifications
    $$P_T\cong S^{(1)}\widehat\otimes_{A\{\underline{x}\},o_1}A\{\underline{y}\},$$
    $$D_T\cong S\widehat\otimes_{A\{\underline{x}\},o_2}A\{\underline{y}\}.$$
    Since $o_2=\phi_{S/A}\circ o_1$, the second ring is also
    $$D_T\cong S\widehat\otimes_{S^{(1)},\phi_{S/A}}P_T.$$
    Under this identification, $\Phi_T$ is the completed base change of $\phi_{S/A}$, and the square
    $$\xymatrix{
    S^{(1)}\ar[r]\ar[d]_{\phi_{S/A}}&P_T\ar[d]^{\Phi_T}\\
    S\ar[r]&D_T
    }$$
    is a pushout square of $p$-complete rings.
    The claim now follows from the fact that $\phi_{S/A}$ is finite faithfully flat.
\end{proof}

For any \(T=(S\twoheadrightarrow B)\in \mathcal{P}_B\), put
$$\Omega_{P_T}:=P_T\otimes_{S^{(1)}}\widehat{\Omega}^1_{S^{(1)}/A},
\qquad
\Omega_{D_T}:=D_T\otimes_S\widehat{\Omega}^1_{S/A}.$$
The map $\Phi_T$ has a divided differential
$$\overline{\dif}\Phi_T:\Omega_{P_T}\longrightarrow \Omega_{D_T}$$
characterized by
$$\dif\Phi_T=p\,\overline{\dif}\Phi_T.$$
Indeed, for $s\in S$ one has
$$\Phi_T(1\otimes s)=\phi_S(s)=s^p+p\delta_S(s),$$
and hence
$$\dif(\Phi_T(1\otimes s))
=p\bigl(s^{p-1}\dif s+\dif(\delta_S(s))\bigr).$$
This determines $\overline{\dif}\Phi_T$ by $P_T$-linearity and $p$-adic continuity.

\begin{defi}[Local lifted inverse Cartier transform]\label{def:local-inverse-cartier}
    Let $T=(S\twoheadrightarrow B)\in \mathcal{P}_B$, and let $(M,\nabla)$ be a finite projective $P_T$-module endowed with an integrable topologically quasi-nilpotent $p$-connection
    $$\nabla:M\to M\otimes_{P_T}\Omega_{P_T}.$$
    Define
    $$C_T^{-1}(M,\nabla):=(D_T\otimes_{\Phi_T,P_T}M,\nabla_\Phi),$$
    where
    $$\nabla_\Phi(b\otimes m)
    =b\cdot(\mathrm{id}_M\otimes\overline{\dif}\Phi_T)(\nabla m)
    +(1\otimes m)\otimes \dif b,\qquad b\in D_T,\ m\in M.$$
    Equivalently, $\nabla_\Phi=\Phi_T^*(\nabla)/p$, where the division by $p$ is taken through $\overline{\dif}\Phi_T$.
\end{defi}

\begin{lemma}\label{lem:inverse-cartier-well-defined}
    The formula in Definition \ref{def:local-inverse-cartier} is well defined on $D_T\otimes_{\Phi_T,P_T}M$ and defines an integrable topologically quasi-nilpotent connection.
\end{lemma}

\begin{proof}
    It is enough to check compatibility with the tensor-product relation. Let $g\in P_T$. Since $\nabla$ is a $p$-connection,
    $$\nabla(gm)=g\nabla(m)+p\,m\otimes \dif g.$$
    Applying $\mathrm{id}_M\otimes\overline{\dif}\Phi_T$ and using $\dif\Phi_T=p\,\overline{\dif}\Phi_T$, we get
    \begin{align*}
        \nabla_\Phi(b\otimes gm)
        &=b\Phi_T(g)(\mathrm{id}_M\otimes\overline{\dif}\Phi_T)(\nabla m)
          +(1\otimes m)\otimes b\,\dif(\Phi_T(g))
          +(1\otimes m)\otimes \Phi_T(g)\dif b\\
        &=\nabla_\Phi(b\Phi_T(g)\otimes m).
    \end{align*}
    The ordinary Leibniz rule over $D_T$ follows immediately from the term $(1\otimes m)\otimes \dif b$.

    We next verify integrability directly. The connection $\nabla_\Phi$ is the sum of the standard connection on the first factor $D_T$ and the $D_T$-linear operator obtained from $\nabla$ by the $P_T$-linear map
    $$\Omega_{P_T}\xrightarrow{\overline{\dif}\Phi_T}\Omega_{D_T}.$$
    Since $\dif\circ\dif\Phi_T=0$ and $\dif\Phi_T=p\,\overline{\dif}\Phi_T$, the map $\overline{\dif}\Phi_T$ is compatible with exterior differentials. Therefore the curvature of the $D_T$-linear part is the image of the curvature of $\nabla$, which is zero. The mixed terms are exactly the exterior differential of the coefficient forms $\overline{\dif}\Phi_T(\omega)$ and cancel by the same compatibility. Hence $\nabla_\Phi^2=0$.

    Finally we check topological quasi-nilpotence using the definition in \cite[\href{https://stacks.math.columbia.edu/tag/07JE}{Tag 07JE}]{stacks-project}. Thus it is enough to show that, for each local section $b\otimes m$, sufficiently high iterates of every local differential operator send $b\otimes m$ into $p(D_T\otimes_{P_T}M)$. Choose local coordinates $t_1,\ldots,t_d$ on $S$ over $A$. Write
    $$\nabla(m)=\sum_i\theta_i(m)\otimes \dif(1\otimes t_i),\qquad
    \overline{\dif}\Phi_T(\dif(1\otimes t_i))=\sum_j c_{ij}\dif t_j,$$
    with $c_{ij}\in D_T$. Then, with respect to the basis $\dif t_j$, the connection operators of $\nabla_\Phi$ have the form
    $$\partial'_j=\partial_j^{D_T}+\sum_i c_{ij}(1\otimes\theta_i).$$
    Let $b\otimes m$ be a local section. By the same definition applied to the standard connection on the $p$-complete PD-envelope $D_T$, for each $j$ there are only finitely many $k$ such that $(\partial_j^{D_T})^k(b)\notin pD_T$. For the $p$-connection on $M$, the operators $\theta_i$ are topologically quasi-nilpotent, so for each $m$ there are only finitely many pairs $(i,k)$ such that $\theta_i^k(m)\notin pM$.

    Expand $(\partial'_j)^N(b\otimes m)$ as a sum over words of length $N$ in the two types of operators
    $$\partial_j^{D_T},\qquad c_{ij}(1\otimes\theta_i).$$
    Fix a bound $N_M$ such that any word in the finitely many operators $\theta_i$ of length at least $N_M$ sends $m$ into $pM$. Such a bound exists because the $\theta_i$ commute by integrability, and for each $i$ some high power $\theta_i^{N_i}$ sends $m$ into $pM$; one can take $N_M> \sum_i(N_i-1)$.
    It remains to consider words containing fewer than $N_M$ operators of the second type. For such a word, the coefficient of the resulting vector is obtained from $b$ and from fewer than $N_M$ of the elements $c_{ij}$ by applying a total of at least $N-N_M$ copies of $\partial_j^{D_T}$.
    By the quasi-nilpotence of the standard connection on the $p$-complete PD-envelope $D_T$ applied to the finite set consisting of $b$ and the $c_{ij}$, this coefficient lies in $pD_T$ once $N$ is sufficiently large. Hence, for $N\gg 0$, every summand of $(\partial'_j)^N(b\otimes m)$ lies in $p(D_T\otimes_{P_T}M)$. Therefore only finitely many pairs $(j,N)$ can satisfy
    $$(\partial'_j)^N(b\otimes m)\notin p(D_T\otimes_{P_T}M),$$
    which is precisely the condition in \cite[\href{https://stacks.math.columbia.edu/tag/07JE}{Tag 07JE}]{stacks-project}.

\end{proof}

For a $T=(S\twoheadrightarrow B)\in \mathcal{P}_B$, let $\mathbf{Strat}_{\Prism}(P_T)$ be the category of finite projective $P_T$-modules with prismatic stratification along the \v{C}ech nerve $P_T^\bullet$, where
$$P_T^n:=\Prism_{B^{(1)}/(S^{(1)})^{[n]}},$$
and $(S^{(1)})^{[n]}$ denotes the $(n+1)$-fold completed tensor product over $A$. Let $\mathbf{Strat}_{\mathrm{crys}}(D_T)$ be the category of finite projective $D_T$-modules with PD-stratification along the \v{C}ech nerve $D_T^\bullet$, where
$$D_T^n:=D_{S^{[n]}}(B).$$
Functoriality of $\Phi$ for the presentations $S^{[n]}\to B$ gives maps
$$\Phi_T^n:P_T^n\longrightarrow D_T^n$$
compatible with the cosimplicial structures.

\begin{defi}[Stratified inverse Cartier transform]\label{def:stratified-inverse-cartier}
    Define
    $$C^{-1}_{T,\mathrm{str}}:\mathbf{Strat}_{\Prism}(P_T)\longrightarrow \mathbf{Strat}_{\mathrm{crys}}(D_T)$$
    as follows. If $(M,\epsilon)$ is a prismatic stratification, set
    $$C^{-1}_{T,\mathrm{str}}(M):=D_T\otimes_{\Phi_T,P_T}M.$$
    The PD-stratification on this module is obtained by base changing $\epsilon$ along the maps $\Phi_T^n:P_T^n\to D_T^n$ for all $n$. The cocycle condition follows from the cocycle condition for $\epsilon$ and the cosimplicial compatibility of the maps $\Phi_T^n$.
\end{defi}

\begin{lemma}\label{lem:cartier-stratification-compatible}
    Let $(M,\epsilon)$ be an object of $\mathbf{Strat}_{\Prism}(P_T)$, and let $(M,\nabla)$ be the associated integrable topologically quasi-nilpotent $p$-connection. Let $\epsilon_\Phi$ be the PD-stratification on $D_T\otimes_{\Phi_T,P_T}M$ obtained from Definition \ref{def:stratified-inverse-cartier}. Under the crystalline-side equivalence between PD-stratifications and integrable topologically quasi-nilpotent connections \cite[\href{https://stacks.math.columbia.edu/tag/07JG}{Tag 07JG}, \href{https://stacks.math.columbia.edu/tag/07JH}{Tag 07JH}]{stacks-project}, the connection associated to $\epsilon_\Phi$ is exactly the connection $\nabla_\Phi$ of Definition \ref{def:local-inverse-cartier}.
    Consequently, Definitions \ref{def:local-inverse-cartier} and \ref{def:stratified-inverse-cartier} define the same functor for the fixed presentation $T$.
\end{lemma}

\begin{proof}
    Let $(M,\nabla)$ be the $p$-connection corresponding to a prismatic stratification $\epsilon$ on $M$. The stratified construction gives a PD-stratification on $D_T\otimes_{\Phi_T,P_T}M$ by base changing $\epsilon$ along the functorial maps $\Phi$ for the two-fold presentation $S^{[1]}\to B$.

    Both the connection $\nabla_\Phi$ of Definition \ref{def:local-inverse-cartier} and this base-changed stratification are integrable and topologically quasi-nilpotent, so it suffices to compare their first-order parts. Let $J_{\Prism}$ be the diagonal ideal of the first prismatic neighborhood of $P_T$, and let $J_{\mathrm{crys}}$ be the diagonal PD-ideal of the first crystalline neighborhood of $D_T$. Under the usual identifications
    $$J_{\Prism}/J_{\Prism}^{[2]}\cong \Omega_{P_T},
    \qquad
    J_{\mathrm{crys}}/J_{\mathrm{crys}}^{[2]}\cong \Omega_{D_T},$$
    the map induced by the divided Frobenius is exactly
    $$\overline{\dif}\Phi_T:\Omega_{P_T}\to\Omega_{D_T}.$$
    Hence the first-order part of the base-changed stratification sends $b\otimes m$ to
    $$b\cdot(\mathrm{id}_M\otimes\overline{\dif}\Phi_T)(\nabla m)
    +(1\otimes m)\otimes \dif b,$$
    which is precisely the formula defining $\nabla_\Phi$. This proves the compatibility.
\end{proof}

\begin{theo}\label{localOV}
    For every affine $X=\Spec(B)$ as above and every presentation $T=(S\twoheadrightarrow B)\in\mathcal P_B$, the functor
    $$C^{-1}_{T,\mathrm{str}}:\mathbf{Strat}_{\Prism}(P_T)\longrightarrow \mathbf{Strat}_{\mathrm{crys}}(D_T)$$
    is an equivalence of categories. Equivalently, for a presentation $T=(R\twoheadrightarrow B)$ and $\a=\ker(R\to B)$, the local lifted inverse Cartier transform 
    \[C^{-1}:
    p\text{-}\mathbf{MIC}^{\nil}_{\Prism_{R^{(1)}}(\a R^{(1)})}
    \longrightarrow
    \mathbf{MIC}^{\nil}_{D_R(\a)}\]
    is an equivalence between topologically quasi-nilpotent $p$-connections on $P_T=\Prism_{R^{(1)}}(\a R^{(1)})$ and topologically quasi-nilpotent connections on $D_T=D_R(\a)$.
\end{theo}

\begin{proof}
    Let $(N,\eta)$ be an object of $\mathbf{Strat}_{\mathrm{crys}}(D_T)$. Put $$D_T'{}^n:=D_T\widehat\otimes_{P_T}\cdots\widehat\otimes_{P_T}D_T
    \qquad (n+1\text{ factors}).$$

    First, we construct a cosimplicial homomorphism
    $$u_T^\bullet:D_T^\bullet\longrightarrow D_T'{}^\bullet.$$
    Using the pushout description in Lemma \ref{lem:presentation-frobenius-map}, we identify
    $$D_T'{}^n\cong
    S\widehat\otimes_{S^{(1)}}\cdots\widehat\otimes_{S^{(1)}}S
    \widehat\otimes_{S^{(1)}}P_T,$$
    with $n+1$ copies of $S$. Equivalently, grouping the last copy of $S$ with $P_T$ gives
    $$D_T'{}^n\cong
    S\widehat\otimes_{S^{(1)}}\cdots\widehat\otimes_{S^{(1)}}S
    \widehat\otimes_{S^{(1)}}D_T,$$
    with $n$ copies of $S$ before the final factor $D_T$.
    Let
    $$I_n:=\ker\left(S^{\widehat\otimes_{S^{(1)}}(n+1)}\to S\right)$$
    be the relative Frobenius diagonal ideal. This ideal has the standard $p$-adically continuous PD-structure: in an \'etale toric chart it is generated by the relative Frobenius diagonal differences, whose divided powers are $p$-adically integral, and the assertion follows by \'etale descent from this chart calculation. Hence $I_nD_T'{}^n$ is a PD-ideal.

    Let
    $$\Delta_n:D_T'{}^n\to D_T$$
    be the diagonal map. If $J_T\subset D_T$ denotes the PD-ideal defining $B$, then under the second description of $D_T'{}^n$ one has
    \begin{equation}\label{eq:preimage-pd-ideal-T}
    \Delta_n^{-1}(J_T)=
    I_nD_T'{}^n+
    S\widehat\otimes_{S^{(1)}}\cdots\widehat\otimes_{S^{(1)}}S
    \widehat\otimes_{S^{(1)}}J_T,
    \end{equation}
    with $n$ copies of $S$ before the final factor $J_T\subset D_T$. Indeed, after quotienting by $I_nD_T'{}^n$, the diagonal identifies the quotient with $D_T$, and the second summand maps onto $J_T$. The first summand in \eqref{eq:preimage-pd-ideal-T} is a PD-ideal by the preceding paragraph, and the second is the completed base change of the PD-ideal $J_T\subset D_T$. Their sum is a PD-ideal by the usual formula
    $$\gamma_m(x+y)=\sum_{a+b=m}\gamma_a(x)\gamma_b(y).$$
    Therefore the composite
    $$S^{[n]}\longrightarrow D_T'{}^n$$
    sends the kernel of $S^{[n]}\to B$ into a PD-ideal. By the universal property of the $p$-complete PD-envelope $D_T^n$, it extends uniquely to
    $$u_T^n:D_T^n\to D_T'{}^n.$$
    The uniqueness gives compatibility with coface and codegeneracy maps, so the $u_T^n$ form $u_T^\bullet$.

    Next, we descend the underlying module. Pulling the PD-stratification $\eta$ back along $u_T^\bullet$ gives a faithfully flat descent datum on the $D_T$-module $N$ relative to the finite faithfully flat map $\Phi_T:P_T\to D_T$. Hence there is a finite projective $P_T$-module $M$, unique up to unique isomorphism, and an isomorphism
    $$D_T\otimes_{\Phi_T,P_T}M\xrightarrow{\sim}N$$
    identifying the canonical descent datum with the pulled-back stratification.

    We now descend the PD-stratification. For every $n$, apply the same construction to the presentation $S^{[n]}\to B$. The map
    $$\Phi_T^n:P_T^n\to D_T^n$$
    is finite faithfully flat by Lemma \ref{lem:presentation-frobenius-map}. The pullback of $N$ to $D_T^n$ descends uniquely to a finite projective $P_T^n$-module. By uniqueness of faithfully flat descent and compatibility with the coface maps, this descended module is the pullback of $M$ to $P_T^n$. The isomorphisms defining $\eta$ therefore descend to isomorphisms between the corresponding pullbacks of $M$ over $P_T^1$, and the cocycle condition over $P_T^2$ follows after the faithfully flat base change to $D_T^2$. Thus $M$ acquires a prismatic stratification $\epsilon$.

    By construction, $C^{-1}_{T,\mathrm{str}}(M,\epsilon)\cong(N,\eta)$, so the functor is essentially surjective. Full faithfulness is the same descent argument applied to morphisms: a morphism between two objects after applying $C^{-1}_{T,\mathrm{str}}$ is a $D_T$-linear map compatible with the pulled-back descent data along $D_T'{}^\bullet$, hence descends uniquely to a $P_T$-linear map. Compatibility with stratifications descends because it can be checked after the faithfully flat base change $\Phi_T^n$ in each cosimplicial degree. Therefore $C^{-1}_{T,\mathrm{str}}$ is an equivalence.

    The equivalence with the connection-theoretic statement follows from Lemma \ref{lem:cartier-stratification-compatible}.
\end{proof}

\begin{proof}[Proof of Theorem \ref{equiv}]
    The construction above is local on $X$. Choose an affine open $\Spec(B)\subset X$ and a presentation $T=(S\twoheadrightarrow B)$ in $\mathcal P_B$. Theorem \ref{localOV} gives a canonical equivalence
    $$\mathbf{Crys}_{(B^{(1)}/A)_{\Prism}}\xrightarrow{\sim}\mathbf{Crys}_{(B/A)_{\mathrm{crys}}}.$$
    This equivalence is independent of the chosen presentation. If $T=(S\twoheadrightarrow B)$ and $T'=(S'\twoheadrightarrow B)$ are two choices, then
    $$T''=(S\widehat\otimes_A S'\twoheadrightarrow B)$$
    dominates both. The functoriality of the maps $\Phi_T$ and of the descent construction identifies the equivalences obtained from $T$ and $T'$ with the one obtained from $T''$.

    The equivalence is also compatible with restriction to principal opens. If $f\in B$ and $\tilde f\in S$ is a lift, then the $p$-complete localization $S\langle \tilde f^{-1}\rangle$ is again a smooth $p$-complete $\delta$-$A$-algebra presenting $B_f$. The prismatic envelopes, PD-envelopes, the maps $\Phi_T$, and the descent data are compatible with this localization. Hence the local equivalences agree on overlaps.

    We may therefore glue the equivalences over an affine open cover of $X$. Since the Frobenius twist is compatible with restriction to affine opens and the local-complete-intersection condition is preserved by the lemma at the beginning of this subsection, the glued functor gives a canonical equivalence
    $$\mathbf{Crys}_{(X^{(1)}/A)_{\Prism}}\cong \mathbf{Crys}_{(X/A)_{\mathrm{crys}}}.$$
\end{proof}
\section{Globalization}\label{Global}
Throughout this section, let $(A,\delta_A)$ be a crystalline prism and put $A_0=A/p$. Let $\X$ be a separated $A_0$-scheme of finite type. We assume that $\X$ admits a Koszul-regular closed immersion
$$\X\hookrightarrow Y$$
into a smooth formal $A$-scheme $Y$, and that $Y\times_{\Spec(A)}\Spec(A/p^2)$ is equipped with a Frobenius lift compatible with $\phi_A$.
\subsection{The ring sheaf $\Prism_Y(\X)$}
\begin{lemma}\label{globalCech}
    Let $(D,\delta)$ be a crystalline prism and let $f\in D$. Then the $p$-complete localization
    $$\widehat D_f:=D\left[\frac{1}{f}\right]^\wedge_p$$
    carries a unique $\delta$-structure extending the one on $D$. In the sequel we regard $\widehat D_f$ as a crystalline prism with this $\delta$-structure.
\end{lemma}
\begin{proof}
    Since $D$ is $p$-torsion free, so is $D[1/f]$. The Frobenius lift $\phi$ on $D$ extends to $D[1/f]$ because
    $$\phi(f)=f^p+p\delta(f)=f^p\left(1+\frac{p\delta(f)}{f^p}\right)$$
    is invertible in the $p$-complete localization: the second factor is a $p$-adic unit. Passing to the $p$-adic completion gives a continuous Frobenius lift on $\widehat D_f$. Since $\widehat D_f$ is $p$-torsion free, the formula
    $$\delta(x)=\frac{\phi(x)-x^p}{p}$$
    uniquely defines a compatible $\delta$-structure on $\widehat D_f$.
\end{proof}
For a $p$-complete ring $S$ and an element $f\in S$, we write
$$\widehat S_f:=S\left[\frac{1}{f}\right]^\wedge_p$$
for the $p$-complete localization. This construction depends only on the image of $f$ modulo $p$; hence $\widehat S_f$ is well defined when $f$ is specified as an element of $S/p$.
\begin{lemma}\label{localizationprism}
    Let $(R,\delta)$ be a crystalline prism, and let $\a\subset R$ be an ideal containing $p$. For every $f\in R$, there is a canonical isomorphism
    $$\widehat{\Prism_R(\a)}_f\cong \Prism_{\widehat R_f}(\a\widehat R_f).$$
\end{lemma}
\begin{proof}
    Both sides satisfy the same universal property. Namely, for any crystalline prism $(B,\delta_B)$, a map from either side to $B$ is equivalent to a $\phi$-compatible map $R\to B$ which sends $\a$ into $pB$ and sends $f$ to a unit of $B$; the latter condition is the universal property of $p$-complete localization. This canonically identifies the two completed prismatic envelopes.
\end{proof}
This localization result allows us to define the ring sheaf $\Prism_Y(\X)$ on $\X$ rather than on $Y$.
\begin{prop}\label{prismaticlocal}
    Let $U=\Spf(R)$ be an affine formal open subset of $Y$, and let $I\subset R$ be the ideal defining $\X\cap U$. Assume that the Frobenius lift on $Y/p^2$ lifts to a Frobenius lift $\phi:R\to R$.
    \begin{enumerate}
        \item[(1)] Let $f_1,\dots,f_r\in R$ satisfy
        $$f_1R+\cdots+f_rR+I=R.$$
        Then
        $$0\to \Prism_R(I)\to \prod_{i=1}^r\Prism_{\widehat R_{f_i}}(I\widehat R_{f_i})\rightrightarrows \prod_{i,j=1}^r\Prism_{\widehat R_{f_if_j}}(I\widehat R_{f_if_j})$$
        is an equalizer diagram, where the two arrows are the usual \v{C}ech restriction maps.
        \item[(2)] If $U'=\Spf(R')$ is another affine formal open subset of $Y$ such that $\X\cap U=\X\cap U'$, and if $I'\subset R'$ defines this closed subscheme, then there is a canonical isomorphism
        $$\Prism_R(I)\cong \Prism_{R'}(I').$$
    \end{enumerate}
\end{prop}
\begin{proof}
    For (1), choose elements $f_{r+1},\dots,f_s\in I$ such that
    $$f_1R+\cdots+f_sR=R.$$
    By Lemma \ref{localizationprism}, the usual sheaf property for the affine formal cover defined by $f_1,\dots,f_s$ gives the equalizer diagram
    $$0\to \Prism_R(I)\to \prod_{i=1}^s\Prism_{\widehat R_{f_i}}(I\widehat R_{f_i})\rightrightarrows \prod_{i,j=1}^s\Prism_{\widehat R_{f_if_j}}(I\widehat R_{f_if_j}).$$
    If $g\in I$, then $I\widehat R_g=\widehat R_g$, so the prismatic envelope $\Prism_{\widehat R_g}(I\widehat R_g)$ is the zero ring. Removing these zero terms gives the asserted equalizer for $f_1,\dots,f_r$.
    For (2), choose principal open covers of $U\cap\X=U'\cap\X$ which can be written both as $\Spf(\widehat R_{f_i})\cap\X$ and as $\Spf(\widehat R'_{g_i})\cap\X$. By Lemma \ref{localizationprism}, the corresponding localized prismatic envelopes are canonically isomorphic. Applying (1) to the two \v{C}ech diagrams shows that $\Prism_R(I)$ and $\Prism_{R'}(I')$ are equalizers of canonically isomorphic diagrams, hence are canonically isomorphic.
\end{proof}
\begin{theo}\label{prismaticglobal}
    There is a canonically defined sheaf of rings $\Prism_Y(\X)$ on the underlying topological space of $\X$ with the following properties.
    \begin{enumerate}
        \item[(1)] For every affine formal open $U=\Spf(R)\subset Y$ on which the Frobenius lift admits a lift $\phi:R\to R$, one has
        $$\Gamma(\X\cap U,\Prism_Y(\X))=\Prism_R(I),$$
        where $I\subset R$ is the ideal defining $\X\cap U$.
        \item[(2)] For every affine open $V\subset\X$, the ring $\Gamma(V,\Prism_Y(\X))$ is $p$-complete, and $\Prism_Y(\X)/p$ is a quasi-coherent $\inte_\X$-algebra. More explicitly, for every $f\in\Gamma(V,\inte_\X)$ there is a canonical isomorphism
        $$\widehat{\Gamma(V,\Prism_Y(\X))}_f\cong \Gamma(V_f,\Prism_Y(\X)).$$
        \item[(3)] For every affine open $V\subset\X$, the natural map
        $$\Gamma(V,\Prism_Y(\X))/p\to \Gamma(V,\Prism_Y(\X)/p)$$
        is an isomorphism.
    \end{enumerate}
\end{theo}
\begin{proof}
    Choose an affine formal open cover
    $$Y=\bigcup_{\lambda\in\Lambda}U_\lambda,\qquad U_\lambda=\Spf(R_\lambda),$$
    such that the given Frobenius lift modulo $p^2$ lifts on each $R_\lambda$ to a Frobenius lift $\phi_\lambda:R_\lambda\to R_\lambda$. Let $I_\lambda\subset R_\lambda$ be the ideal defining $\X\cap U_\lambda$.
    On each $\X\cap U_\lambda$ we first define a sheaf, still denoted $\Prism_{R_\lambda}(I_\lambda)$. For a principal open $V=\Spec((R_\lambda/I_\lambda)_{\bar f})\subset \X\cap U_\lambda$, choose a lift $f\in R_\lambda$ of $\bar f$ and set
    $$\Gamma(V,\Prism_{R_\lambda}(I_\lambda)):=\Prism_{\widehat{R_{\lambda,f}}}(I_\lambda\widehat{R_{\lambda,f}}).$$
    Lemma \ref{localizationprism} and Proposition \ref{prismaticlocal} show that this definition is independent of the choice of $f$ and satisfies the sheaf condition on the basis of principal opens. Thus it defines a sheaf on $\X\cap U_\lambda$ whose global sections are $\Prism_{R_\lambda}(I_\lambda)$.
    It remains to glue these local sheaves. Let $\Spf(C)\subset U_{\lambda_1}\cap U_{\lambda_2}$ be an affine formal open, and let $J\subset C$ be the ideal defining $\X\cap\Spf(C)$. Let $\phi_1,\phi_2:C\to C$ be the restrictions of the two local Frobenius lifts. Since both lift the same Frobenius modulo $p^2$, we have
    $$\phi_1\equiv\phi_2\pmod{p^2},\qquad\text{equivalently}\qquad \delta_1\equiv\delta_2\pmod p.$$
    Corollary \ref{closure mod p} gives a canonical isomorphism
    $$\Prism_{C,\phi_1}(J)\cong \Prism_{C,\phi_2}(J).$$
    By Proposition \ref{prismaticlocal}, these are precisely the restrictions of the two local sheaves to $\X\cap\Spf(C)$. Hence the displayed isomorphisms provide descent data on overlaps. The cocycle condition on triple overlaps follows from the canonicity in Corollary \ref{closure mod p}: all transition maps are induced by identifying the same $p$-complete subring of $C[1/p]$ generated by the elements $\delta^n(x/p)$.
    Therefore the local sheaves glue to a sheaf $\Prism_Y(\X)$ on $\X$. Property (1) is built into the construction. Property (2) follows from the localization isomorphism of Lemma \ref{localizationprism}, and property (3) follows because the sheaf was defined on the basis of principal opens by $p$-complete rings whose reductions modulo $p$ are compatible with the same principal localization.
\end{proof}
Taking relative formal spectra, we obtain the following corollary.
\begin{cor}
    There is a formal scheme, still denoted $\Prism_Y(\X)$, affine over $Y$, together with an affine morphism $\Prism_Y(\X)/p\to \X$, such that for every affine formal open $U=\Spf(R)\subset Y$ admitting a Frobenius lift, the preimage of $U$ is canonically isomorphic to $\Spf(\Prism_R(I))$, where $I\subset R$ defines $\X\cap U$. Moreover, the sheaf $\Prism_Y(\X)$ on $\X$ is the pushforward of the structure sheaf of this formal scheme.
\end{cor}
\begin{defi}
    Define
    $$\Omega_{\Prism_Y(\X)}^n:=\pi_*\pi^*\Omega_{Y/A}^n,$$
    where $\pi:\Prism_Y(\X)\to Y$ is the natural morphism. This is a locally free $\Prism_Y(\X)$-module, and the local $p$-differentials glue to maps
    $$p\dif:\Omega_{\Prism_Y(\X)}^n\to\Omega_{\Prism_Y(\X)}^{n+1}.$$
\end{defi}
\subsection[An independence property of D-Prism]{An independence property of $\D_\Prism$}\label{GlobalD}
We use the notation of Section \ref{Dprism}. Let $R$ be an affine local chart as above, and assume that $R$ carries two $\delta$-structures $\delta_1,\delta_2$ such that
$$\delta_1\equiv\delta_2\pmod p.$$
Let
$$P_i:=\Prism_{R,\delta_i}(\a),\qquad i=1,2.$$
By Corollary \ref{closure mod p}, $P_1$ and $P_2$ are canonically isomorphic as rings. Fix a prismatic crystal $\mathbb H$ on $(B/A)_{\Prism}$, and set $M_i:=\mathbb H(P_i)$.
For $i,j\in\{1,2\}$, let $P_{i,j}$ be the coproduct of $P_i$ and $P_j$ in $(B/A)_{\Prism}$, and set $M_{i,j}:=\mathbb H(P_{i,j})$. We use the analogous notation for triple coproducts. Since the product $\delta$-structures $\delta_i\otimes\delta_j$ are congruent modulo $p$, all rings $P_{i,j}$ with the same number of factors are canonically identified. In particular, we have a canonical isomorphism
$$\psi:M_{1,2}\xrightarrow{\sim}M_{1,1},$$
defined as the composite
$$M_{1,2}\cong P_{1,2}\otimes_{P_1}M_1
\cong P_{1,1}\otimes_{P_1}M_1
\cong M_{1,1},$$
where $P_1\to P_{1,2}$ and $P_1\to P_{1,1}$ are the maps into the first factor. The maps $P_2\to P_{1,2}$ and $P_1\to P_{1,1}$ into the second factor induce
$$\tau:M_2\to M_{1,2},\qquad \tau':M_1\to M_{1,1}.$$
\begin{rmk}
    If $a_1,\dots,a_h$ is a sequence with values in $\{1,\dots,k\}$, we write $P_{a_1,\dots,a_h}$ for the corresponding iterated coproduct. If $1\leq j_1<\cdots<j_s\leq h$, we write
    $$\nu_{*,\dots,*,a_{j_1},*,\dots,*,a_{j_s},*,\dots,*}:P_{a_{j_1},\dots,a_{j_s}}\to P_{a_1,\dots,a_h}$$
    for the map sending the $t$-th factor to the $j_t$-th factor. We use the same convention for the induced maps
    $$\eta_{*,\dots,*,a_{j_1},*,\dots,*,a_{j_s},*,\dots,*}:M_{a_{j_1},\dots,a_{j_s}}\to M_{a_1,\dots,a_h}.$$
\end{rmk}
\begin{theo}\label{globalDindependence}
    With the notation above:
    \begin{enumerate}
        \item[(1)] The maps $\tau$ and $\tau'$ are injective, and
        $$\psi\circ\tau(M_2)=\tau'(M_1)$$
        as submodules of $M_{1,1}$.
        \item[(2)] The map
        $$\iota_{2,1}:=\tau'^{-1}\circ\psi\circ\tau:M_2\to M_1$$
        is $P_2$-linear, where the $P_2$-module structure on $M_1$ is induced by the canonical ring isomorphism $P_2\cong P_1$.
        \item[(3)] The map $\iota_{2,1}$ is compatible with the stratifications attached to $M_2$ and $M_1$.
        \item[(4)] If $\delta_3$ is a third $\delta$-structure with $\delta_1\equiv\delta_2\equiv\delta_3\pmod p$, then
        $$\iota_{2,1}\circ\iota_{3,2}=\iota_{3,1}.$$
    \end{enumerate}
\end{theo}
\begin{proof}
    The injectivity of $\tau$ and $\tau'$ follows from the separatedness of the sheaf $\mathbb H$ applied to the covers $P_{1,2}\to P_2$ and $P_{1,1}\to P_1$. Let
    $$\beta:P_{1,2}\xrightarrow{\sim}P_{1,1},\qquad
    \beta_*:P_2\xrightarrow{\sim}P_1,\qquad
    \beta':P_{1,1,2}\xrightarrow{\sim}P_{1,1,1}$$
    be the canonical isomorphisms coming from Corollary \ref{closure mod p}. These isomorphisms are compatible with all insertions of factors; for instance
    $$\nu_{*,1,1}\circ\beta=\nu_{*,1,2}\circ\beta',
    \qquad
    \nu_{1,*,1}\circ\beta=\nu_{1,*,2}\circ\beta',
    \qquad
    \beta\circ\nu_{*,2}=\nu_{*,1}\circ\beta_*.$$
    Thus we have a morphism between the two equalizer diagrams defining the sheaf condition:
    $$\xymatrix{
    0\ar[r]& P_2\ar[d]^-{\beta_*}\ar[r]&P_{1,2}\ar[d]^-{\beta}\ar@<0.5ex>[r]\ar@<-0.5ex>[r]&P_{1,1,2}\ar[d]^-{\beta'}\\
    0\ar[r]& P_1\ar[r]&P_{1,1}\ar@<0.5ex>[r]\ar@<-0.5ex>[r]&P_{1,1,1}.
    }$$
    Applying the crystal $\mathbb H$ and using the crystal identifications gives the corresponding diagram for the $M$'s. It follows that $\psi\circ\tau(M_2)$ and $\tau'(M_1)$ are the same equalizer inside $M_{1,1}$. This proves (1).
    The same diagram also proves $P_2$-linearity of $\iota_{2,1}$ after identifying $P_2$ with $P_1$ through $\beta_*$. This proves (2).
    For (3), consider the stratification maps obtained by pulling $M_2$ and $M_1$ to $P_{2,2}$ and $P_{1,1}$. The canonical identifications of the coproducts give an isomorphism
    $$M_{2,2}\xrightarrow{\sim}M_{1,1}$$
    obtained through the chain
    $$M_{2,2}\cong P_{2,2}\otimes_{P_2}M_2
    \cong P_{1,2}\otimes_{P_2}M_2
    \cong M_{1,2}
    \cong P_{1,2}\otimes_{P_1}M_1
    \cong M_{1,1}.$$
    The compatibility of the canonical ring isomorphisms with insertion maps gives a commutative diagram
    $$\xymatrix{
    M_2\ar[d]\ar[r]^-{\iota_{2,1}}&M_1\ar[d]\\
    M_{2,2}\ar[r]^-{\sim}&M_{1,1}.
    }$$
    This is exactly the assertion that $\iota_{2,1}$ respects the two stratifications.
    Finally, (4) follows from the same functoriality for triple choices of $\delta$-structures. Over the common triple coproduct $P_{1,2,3}$, the three maps $\iota_{3,2}$, $\iota_{2,1}$, and $\iota_{3,1}$ are all induced by the same canonical identifications of prismatic envelopes. Hence their composites agree.
\end{proof}
\begin{theo}\label{globalD}
    Let $\X$ be a syntomic $A_0$-scheme admitting a closed immersion $\X\hookrightarrow Y$ into a smooth formal $A$-scheme. Assume that a Frobenius lift $\phi_Y$ on $Y/p^2$ compatible with $\phi_A$ is fixed. Then there is a rank-preserving equivalence of categories
    $$\mathbf{Crys}_{(\X/A)_{\Prism}}\cong p\text{-}\mathbf{MIC}^{\nil}_{\Prism_Y(\X)}.$$
\end{theo}
\begin{proof}
    Choose an affine open cover of $\X$ by opens which admit toric charts and local lifts of $\phi_Y$ to the corresponding smooth formal affine pieces of $Y$. On each such open, Theorem \ref{localDprism} gives an equivalence between prismatic crystals and integrable topologically quasi-nilpotent $p$-connections on the local prismatic envelope.
    On an overlap, two choices of Frobenius lift are congruent modulo $p^2$ because they both lift the fixed $\phi_Y$ on $Y/p^2$. Therefore the associated $\delta$-structures are congruent modulo $p$. Theorem \ref{globalDindependence} identifies the corresponding local modules together with their stratifications, hence identifies the local $p$-connections. On triple overlaps, part (4) of Theorem \ref{globalDindependence} gives the cocycle condition. Thus the local equivalences glue to an equivalence over the sheaf $\Prism_Y(\X)$. Rank preservation is local on $\X$, so the glued equivalence is rank preserving.
\end{proof}
\subsection{Cohomological comparison: Globalization}
Let $\X_{\et}$ denote the \'etale site of $\X$. Recall that there is a morphism of topoi
$$\nu:\shv((\X/A)_{\Prism})\to \shv(\X_{\et})$$
constructed in \cite[Section 4]{Bhatt_2022}. We prove the following global comparison.
\begin{theo}\label{globalcohcompare}
    Let $\mathbb H\in\mathbf{Crys}_{(\X/A)_{\Prism}}$, and let $(\mathcal M,\nabla)$ be the corresponding object under the equivalence of Theorem \ref{globalD}. Then there is a canonical quasi-isomorphism
    $$R\nu_*\mathbb H\simeq \DR(\mathcal M,\nabla).$$
\end{theo}
Before the proof, we construct the global \v{C}ech object. For $n\geq 0$, let $P^n$ be the sheaf
$$P^n:=\Prism_{Y^{n+1}}(\X),$$
where $\X$ is embedded in $Y^{n+1}$ by the diagonal map. By functoriality, $P^\bullet$ is a cosimplicial sheaf of rings on $\X$. For a crystal $\mathbb H$, define a cosimplicial $P^\bullet$-module $\mathbb H^\bullet$ as follows. Locally on $Y$, after choosing a lift of $\phi_Y$, $\mathbb H(P^n)$ is the usual evaluation of $\mathbb H$ on the local prismatic envelope. On overlaps, Theorem \ref{globalDindependence} identifies the evaluations for different local Frobenius lifts, compatibly on triple overlaps. Hence these local evaluations glue to a sheaf, denoted $\mathbb H^n$, and the $\mathbb H^n$ form a cosimplicial module $\mathbb H^\bullet$.
\begin{lemma}\label{lem:global-cech-compute}
    There is a canonical quasi-isomorphism
    $$R\nu_*\mathbb H\simeq \mathrm{Tot}(\mathbb H^\bullet).$$
\end{lemma}
\begin{proof}
    The statement is local on $\X_{\et}$. On an affine chart where the Frobenius lift has been chosen, the object $P^0$ is a cover of the final object of the local prismatic topos by Theorem \ref{coverprism}. Therefore its \v{C}ech nerve $P^\bullet$ computes the derived global sections of the crystal. More explicitly, this is the same \v{C}ech--Alexander calculation as \cite[Proposition 3.10]{Tian_2023}: the local prismatic envelope is a cover of the final object, and the higher cohomology of a crystal on each term of the \v{C}ech nerve vanishes, so the \v{C}ech totalization computes $R\nu_*\mathbb H$.
    The local identifications are independent of the chosen lift of Frobenius by Theorem \ref{globalDindependence}. Hence they agree on overlaps and glue to the asserted global quasi-isomorphism.
\end{proof}
\begin{proof}[Proof of Theorem \ref{globalcohcompare}]
    By Lemma \ref{lem:global-cech-compute}, it remains to identify $\mathrm{Tot}(\mathbb H^\bullet)$ with $\DR(\mathcal M,\nabla)$. This can be checked on an affine open of $\X$ admitting a toric chart and a lift of $\phi_Y$. On such an open, Theorem \ref{localcoh} gives a canonical quasi-isomorphism
    $$\mathrm{Tot}(\mathbb H(P^\bullet))\simeq \DR(\mathcal M,\nabla).$$
    The proof of Theorem \ref{localcoh} gives an explicit bicomplex comparison. It is modeled on Tian's \v{C}ech--de Rham calculation and Wang's local calculation: the horizontal \v{C}ech direction computes the crystal by \cite[Proposition 3.10]{Tian_2023}, while the vertical direction contracts to the de Rham complex by the local calculation of Theorem \ref{localtotqi}.
    If two local Frobenius lifts are chosen on an overlap, Theorem \ref{globalDindependence} identifies the corresponding prismatic envelopes, the evaluated modules, and the stratifications defining the $p$-connections. Under these identifications the local bicomplexes used in Theorem \ref{localcoh} agree term by term. Consequently the local quasi-isomorphisms agree on overlaps. The cocycle condition on triple overlaps again follows from part (4) of Theorem \ref{globalDindependence}. Therefore the local quasi-isomorphisms glue to a canonical global quasi-isomorphism
    $$R\nu_*\mathbb H\simeq \DR(\mathcal M,\nabla).$$
\end{proof}

\bibliographystyle{alpha}
\bibliography{ref}

@misc{ogus2023crystalline,
      title={Crystalline prisms: Reflections on the present and past},
      author={Ogus, Arthur},
      year={2023},
      eprint={2204.06621},
      archivePrefix={arXiv},
      primaryClass={math.AG}
}

@article{Bhatt_2022, 
      title={Prisms and prismatic cohomology}, 
      volume={196}, 
      ISSN={0003-486X}, 
      url={http://dx.doi.org/10.4007/annals.2022.196.3.5}, 
      DOI={10.4007/annals.2022.196.3.5}, 
      number={3}, 
      journal={Annals of Mathematics}, 
      publisher={Annals of Mathematics}, 
      author={Bhatt, Bhargav and Scholze, Peter}, 
      year={2022}, 
      month=nov,
      pages={1135--1275}
}

@misc{wang2024prismatic,
      title={Prismatic crystals for smooth schemes in characteristic $p$ with Frobenius lifting mod $p^2$},
      author={Wang, Yupeng},
      year={2024},
      eprint={2402.02109},
      archivePrefix={arXiv},
      primaryClass={math.AG}
}

@article{Bhatt_2023, 
      title={Prismatic ${F}$-crystals and crystalline {G}alois representations}, 
      volume={11}, 
      ISSN={2168-0949}, 
      url={http://dx.doi.org/10.4310/cjm.2023.v11.n2.a3}, 
      DOI={10.4310/cjm.2023.v11.n2.a3}, 
      number={2}, 
      journal={Cambridge Journal of Mathematics}, 
      publisher={International Press of Boston}, 
      author={Bhatt, Bhargav and Scholze, Peter}, 
      year={2023}, 
      pages={507--562} 
}

@article{Tian_2023, 
      title={Finiteness and duality for the cohomology of prismatic crystals}, 
      volume={799}, 
      ISSN={1435-5345}, 
      url={http://dx.doi.org/10.1515/crelle-2023-0032}, 
      DOI={10.1515/crelle-2023-0032}, 
      journal={Journal f\"ur die reine und angewandte Mathematik (Crelles Journal)}, 
      publisher={Walter de Gruyter GmbH}, 
      author={Tian, Yichao}, 
      year={2023}, 
      pages={217--257} 
}

@misc{stacks-project,
      author       = {The {Stacks project authors}},
      title        = {The Stacks project},
      howpublished = {\url{https://stacks.math.columbia.edu}},
      year         = {2026},
}

@misc{bhatt2022absolute,
      title={Absolute prismatic cohomology}, 
      author={Bhargav Bhatt and Jacob Lurie},
      year={2022},
      eprint={2201.06120},
      archivePrefix={arXiv},
      primaryClass={math.AG},
      url={https://arxiv.org/abs/2201.06120}
}

@misc{bhatt2011crystalline,
  title={Crystalline cohomology and de Rham cohomology},
  author={Bhatt, Bhargav and de Jong, Aise Johan},
  year={2011},
  eprint={1110.5001},
  archivePrefix={arXiv},
  primaryClass={math.AG},
  url={https://arxiv.org/abs/1110.5001}
}

\end{document}